\documentclass[11pt,a4paper,twoside]{amsart}

\usepackage{amsmath,amssymb,amsthm}
\usepackage{inputenc}
\usepackage{tikz} 
\usepackage{float}
\floatstyle{boxed} 
\usepackage{graphicx,subfigure}
\usepackage{caption}
\usepackage[T1]{fontenc} 
\usepackage{enumitem}
\usepackage{mathtools}
\usepackage{mathrsfs}
\usepackage{xcolor}
\usepackage[colorlinks=true, linkcolor=magenta, citecolor=blue, urlcolor=blue,hyperfootnotes=false]{hyperref}
\usepackage{cleveref}
\usepackage{etoolbox}
\usepackage{upgreek}

\def\@begintheorem#1#2{\par\bgroup{\sc #1 \ #2. }  \it \\\ignorespace }
\def\@opargbegintheorem#1#2#3{\par\bgroup{\sc #1\ #2 \ (#3).}  \it  \ignorespace}
\def\@endtheorem{\egroup}

\theoremstyle{plain}
\newtheorem{theorem}{Theorem}[section]

\theoremstyle{definition}

\theoremstyle{remark}
\newtheorem{remark}[theorem]{Remark}

\theoremstyle{plain}

\theoremstyle{plain}
\newtheorem{lemma}[theorem]{Lemma}

\theoremstyle{plain}
\newtheorem{corollary}[theorem]{Corollary}

\theoremstyle{plain}
\newtheorem{proposition}[theorem]{Proposition}

\numberwithin{equation}{section}

\newcommand{\C}{{\mathbb C}}

\newcommand{\N}{{\mathbb N}}
\newcommand{\D}{\mathcal{D}}

\newcommand{\R}{{\mathbb R}}

\newcommand{\I}{\mathbb{I}}

\newcommand{\V}{\mathbb{V}}

\newcommand{\Rt}{{\R}^3}
\newcommand{\hx}{\hat{x}}
\renewcommand{\H}{\mathcal{H}}

\newcommand{\de}{\partial}

\DeclareMathOperator{\Realpart}{Re}
\renewcommand{\Re}{\Realpart}
\DeclareMathOperator{\Imaginarypart}{Im}
\renewcommand{\Im}{\Imaginarypart}
{\left\lbrace\begin{array}{@{}l@{}}}%
{\end{array}\right.}

\DeclarePairedDelimiter{\abs}{\lvert}{\rvert}
\DeclarePairedDelimiter{\norm}{\lVert}{\rVert}
\DeclarePairedDelimiter{\seq}{\lbrace}{\rbrace}

\newcommand{\Hm}{H_{min}}

\newcommand{\ignora}[1]{}

\title[The Dirac Operator with Coulomb-Type Spherically Symmetric Potentials]{Self-Adjoint Extensions for the Dirac Operator \\with Coulomb-Type Spherically Symmetric Potentials}
\date{\today}
\author[B. Cassano, F. Pizzichillo]{Biagio Cassano and Fabio Pizzichillo}
\subjclass[2010]{Primary 81Q10; Secondary 47N20,  47N50, 47B25.}
\keywords{Dirac operator, Coulomb potential, Hardy inequality, self-adjoint operator.}
\address{B. Cassano, BCAM - Basque Center for Applied Mathematics,
Alameda de Mazarredo 14, 48009 Bilbao (Spain)}
\email{bcassano@bcamath.org}
\address{F. Pizzichillo, BCAM - Basque Center for Applied Mathematics,
Alameda de Mazarredo 14, 48009 Bilbao (Spain)}
\email{fpizzichillo@bcamath.org}

\begin{document}
\begin{abstract}
We describe the self-adjoint realizations of the operator 
$H:=-i\alpha\cdot \nabla + m\beta + \mathbb V(x)$, for $m\in\mathbb R $, 
and 
$\mathbb V(x)= \abs{x}^{-1} ( \nu \mathbb{I}_4 +\mu \beta -i
  \lambda \alpha\cdot{x}/{\abs{x}}\,\beta)$, for $\nu,\mu,\lambda \in \mathbb R$.
We characterize the self-adjointness in terms of the behaviour of the functions of the domain in the origin,
exploiting Hardy-type estimates and trace lemmas.
Finally, we describe the \emph{distinguished} extension.
\end{abstract}
\maketitle

\section{Introduction and main results}
In this paper we are interested in 
the self-adjoint realizations of the 
differential operator $H:=H_0+\V$, 
where $H_0$ is the free Dirac operator in $\R^3$ defined by 
\begin{equation*}
H_0:=-i\alpha\cdot\nabla+m\beta, 
\end{equation*}
where $m\in\R$,
\begin{gather*}
\beta=\left(\begin{array}{cc}
\mathbb{I}_2&0\\
0&-\mathbb{I}_2
\end{array}\right),
\quad 
\mathbb{I}_2:=\left(
\begin{array}{cc}
1 & 0\\
0 & 1
\end{array}\right), \\
\alpha=(\alpha_1,\alpha_2,\alpha_3),
\quad
\alpha_j=\left(\begin{array}{cc}
0& {\sigma}_j\\
{\sigma}_j&0
\end{array}\right)\quad (j=1,2,3),
\end{gather*}
$\sigma_k$ are the \emph{Pauli matrices}
\begin{equation*}\label{paulimatrices}
\quad{\sigma}_1 =\left(
\begin{array}{cc}
0 & 1\\
1 & 0
\end{array}\right),\quad {\sigma}_2=\left(
\begin{array}{cc}
0 & -i\\
i & 0
\end{array}
\right),\quad{\sigma}_3=\left(
\begin{array}{cc}
1 & 0\\
0 & -1
\end{array}\right),
\end{equation*}
and 
\begin{equation}\label{eq:def.V}
\V(x):= \frac{1}{|x|}\left( \nu \mathbb{I}_4 +\mu \beta +
  \lambda \left(-i\alpha\cdot\frac{x}{\abs{x}}\,\beta \right)\right), \quad  \text{for } x\neq 0,  
\end{equation}
where $\nu$, $\lambda$ and $\mu$ are real numbers, 
and $\mathbb{I}_4$ is the 
$4 \times 4$ identity matrix.

The \emph{Coulomb potential} $\V_C$ is defined as 
\begin{equation*}
\V_{C}(x)=\frac{\nu}{|x|}\I_4,
\end{equation*}
with $\nu=e^2Z/\hbar$, 
where 
$Z$ is the atomic number, $e$ is the charge of the electron and $\hbar$ is the Plank's constant (we set $\hbar=1$).
The operator $H_0 + \mathbb V_C$ describes relativistic
spin--$\frac12$ particles in the external electrostatic field of an atomic nucleus. 

In quantum mechanics, observables correspond to self-adjoint operators.
For this reason, it is physically interesting to study of the self-adjointness of the operator $H_0+\V_{C}$. 
The first contribution was made by Case in \cite{case1950singular}: in this work, the author was the first to observe that some boundary conditions are required at zero. 
Anyway, the first result of self-adjointness is due to Kato in \cite{kato1951fundamental} and it is based on Hardy inequality
\begin{equation}\label{eq:classicalhardy}
\frac{1}{4}\int_{\Rt}\frac{|f|^2}{|x|^2}\, dx\leq \int_{\Rt} |\nabla f|^2\,dx, \quad\text{for}
\
f\in C^\infty_c(\Rt),
\end{equation}
and the Kato-Rellich Theorem.
He could prove that for $|\nu|\in\left[0,\frac12\right)$, the operator $H_0+\V_{C}$ is essentially self-adjoint on $C^\infty_c(\Rt)^4$ and self-adjoint on $\D(H_0)=H^1(\Rt)^4$.
Kato's approach could be used independently on the spherical symmetry of the potential: it is possible to consider a $4\times 4$ Hermitian real-valued matrix potential $\V$ such that
\begin{equation*}
|\V_{i,j}(x)|\leq a\frac{1}{|x|}+b,
\end{equation*}
with $b\in\R$ and $a<1/2$, see \cite[Theorem V 5.10]{kato2013perturbation}.

The result of Kato does not cover the whole range of $\nu$ on which the Dirac-Coulomb operator is essentially self-adjoint.
In fact several different  approaches were developed
in order to expand  the range of admissible $\nu$.
In \cite{rellich1953eigenwerttheorie} by Rellich and in \cite{weidmann1971oszillationsmethoden} by Weidmann,
using the \emph{partial wave decomposition} and the Weyl-Stone theory for systems of ordinary differential equations,
the range $|\nu|\in \left[0,\frac{\sqrt{3}}{2}\right)$ was recovered.
Moreover, generalizing the Kato-Rellich Theorem and by means of the theory of Fredholm operators, Rejt\"o firstly recaptured  the range $\nu\in\left[0,\frac{3}{4}\right)$ in \cite{rejto1971some} and few years later $|\nu|\in \left[0,\frac{\sqrt{3}}{2}\right)$ in \cite{gustafson1973some} with Gustafson.
Finally, in \cite{schmincke1972essential}, Schmincke considered $H_0+\V_C=(H_0+S)+(\V_C-S)$,
being
$S$ a suitable \emph{intercalary} operator. Then, he proved the self-adjointness 
of $H_0+\V$ showing that $H_0+S$ is self-adjoint and $\V_C-S$ is a small perturbation of $H_0+S$, in the sense of the Kato-Rellich Theorem.

This range of $\nu$ such that the operator $H_0+\V_C$ is essentially self-adjoint on $C^\infty_c(\Rt)^4$ is optimal, in fact for $|\nu|>\sqrt{3}/2$ $H_0+\V_C$ is not essentially self-adjoint and several self-adjoint extensions can be constructed. The main interest was the study, among all, of the most physically meaningful extension.
The first work is \cite{schmincke1972distinguished} by Schmincke: for $|\nu|\in\left(\frac{\sqrt{3}}{2},1\right)$ and by means of a multiplicative intercalary operator, he proved that $H_0+\V_C$ admits a unique self-adjoint extention $H_S$ such that
\begin{equation}\label{eq:schminke}
\D(H_S) \subset \D(r^{-1/2})=\seq{\psi\in L^2(\Rt)^4:|x|^{-1/2}\psi\in L^2(\Rt)^4}.
\end{equation}
Another explicit construction of a distinguished self-adjoint extension was made by W\"ust in \cite{wust1975distinguished}: using a cut-off procedure, he built a sequence of  
self-adjoint operators that converges
strongly in the operator graph topology to a self-adjoint extension of  $H_0+\V_C$, whose domain is contained in $\D (r^{-1/2})$.
Moreover in \cite{nenciu1976self}, Nenciu proved the existence of a unique self-adjoint extension of $H_0+\V_C$ whose domain is contained in the Sobolev space $H^{1/2}(\Rt)^4$.
Finally, Klaus and W\"ust showed in \cite{klaus1979characterization} that these self-adjoint extensions coincide.
We also cite \cite{burnap1981dirac}: in this work, using the partial wave decomposition and the Von Neumann theory, the authors could characterize the distinguished self-adjoint extension by the fact that the energy of the ground state is continuous in $\nu$. 
In \cite{gallonemichelangeli2017self}, applying the Kre\u{\i}n-Vi\v{s}ik-Birman extension theory, Gallone and Michelangeli described the self-adjointness of $H_0+\V_C$ for $\nu<1$, in terms of boundary conditions at the origin,
and in \cite{gallone2017discrete} they determine the discrete spectrum of such extensions. 

In \cite{arai1983essential}, Arai considered matrix-valued potentials as in \eqref{eq:def.V}.
Defining
\begin{equation}\label{eq:delta_k}
\delta:=(k+\lambda)^2-\nu^2+\mu^2,\quad \text{for fixed}\ k\in \mathbb{Z}\setminus\seq{0},
\end{equation}
he proved that a necessary and sufficient condition for the essential self-adjointness of $H_0+\V$ is $\delta\geq 1/4$ for any $k$.
This proved that, in the case of general matrix valued potentials, the threshold $1/2$ is optimal for the essential self-adjointness.
For $\delta>0$ for all $k$, he proved that the operator admits infinitely many self-adjoint extensions.
Kato in \cite{kato1983holomorphic} considered a general $4\times 4$ matrix-valued measured function $\V$ such that for any $x\neq 0$,  $|\V_{i,j}(x)|\leq |x|^{-1}$.
Setting $H(\kappa):=H_0+\kappa\V$, he constructed a unique \emph{holomorphic family} of self-adjoint operators for $|\kappa|<1$, which reduced to the self-adjioint operator $H_0+\kappa\V$ defined on $H^1(\Rt)^4$ for $|\kappa|<1/2$.
Moreover he proved that, in the case of $\V=\V_C=\frac{1}{|x|}\I_4$, this family coincides with the distinguished self-adjoint extension defined by W\"ust and Nenciu.
With a similar idea, in \cite{adv2013self} Arrizabalaga, Duoandikoetxea and Vega were able to characterize the distinguished self-adjoint extension by means of the 
\emph{Kato-Nenciu} inequality
\begin{equation*}
\int_{\Rt}\frac{|\psi|^2}{|x|}\,dx \leq \int_{\Rt}\abs*{(-i\alpha\cdot\nabla+m\beta\pm i)\psi}^2|x|\,dx, \quad\text{for}
\ \psi\in C^\infty_c(\Rt)^4.
\end{equation*}

The self-adjointness in the range of critical values $|\nu|\geq 1$ has been the aim of several recent works: in the case of the Coulomb potential and using the spherical symmetry of the potential, with different approaches Xia in \cite{xia1999contribution}, Voronov in \cite{voronov2007dirac}, Hogreve in \cite{hogreve2012overcritical}
could characterize via boundary conditions all the self-adjoint extensions. 
In \cite{estebanloss}, Esteban and Loss could consider a general electrostatic potential, that is  a
function $V:\Rt\to \R$ such that that for some constant $c(V)\in (-1,1)$, $\Gamma:=\sup(V)<1+c(V)$ and for every $\varphi\in C^\infty_c(\Rt,\C^2)$,
\begin{equation}\label{eq:diseq:EL}
\int_{\Rt}\left( \frac{|\sigma\cdot\nabla \varphi|^2}{1+c(V)-V}+\left(1+c(V)+V\right)|\varphi|^2\right) dx\geq 0.
\end{equation}
Setting $\V:= V \I_4$, they proved that the operator $H_0+\V$ is self-adjoint on a suitable domain.
Although the free Dirac
operator is not semi-bounded,
they defined a \emph{reduced} operator acting only on the two first components of
the wave function, for which the Friedrichs extension can be defined thanks the inequality \eqref{eq:diseq:EL}.
Once this is done, they extended the \emph{whole} operator in a straightforward way.
This allows treating all the potentials of the form $V(x)=-\frac{\nu}{|x|}$ for $\nu\in(0,1]$.
In the sub-critical case, i.~e.~$0<\nu<1$, the self-adjoint extension that they described coincides with the distinguished self-adjoint extension given by W\"ust and Nenciu;
in the critical case, i.~e.~$\nu=1$,  they stated that the distinguished the self-adjoint extension that they are describing is the distinguished one since it can be covered by continuous prolongation of the sub-critical case. 
Recently, in \cite{esteban2017domains}, Esteban, Lewin and S\'er\'e have given more properties of this  domain:
they showed that the self-adjoint extension given by Esteban and Loss could be obtained as the limit of the cut-off procedure and, in the Coulomb case, it is the only extension containing the \emph{ground states}.

The aim of this paper is to give a simple and unified approach to the problem of the self-adjointness of 
$H:=H_0 + \V$, with $\V$ as in \eqref{eq:def.V}.
This particular choice of the class of potentials is related to the fact that the action of $H_0+\V$ leaves invariant the \emph{partial wave subspaces}.
In detail, setting
\begin{equation*}
  \mathbb{V} = 
  \mathbb{V}_{el} + 
  \mathbb{V}_{sc} + 
  \mathbb{V}_{am} :=
  v_{el}(x) \mathbb{I}_4 +
  v_{sc}(x) \beta +
  v_{am}(x) \left(-i \alpha\cdot\frac{x}{\abs{x}}\,\beta \right),
\end{equation*}
for real valued $v_{el},v_{sc}, v_{am}$, the potentials $\mathbb{V}_{el}, \mathbb{V}_{sc},  \mathbb{V}_{am}$ are 
said respectively
an \emph{electric}, \emph{scalar}, and \emph{anomalous magnetic} potential. 

The strategy of the proof is to consider the self-adjointness of the reduction of $H_0+\V$ to the \emph{partial wave subspaces} and, using weighted Hardy-type inequalities and trace theorems, we describe the domain of the \emph{maximal operator}, namely the set of functions $\psi\in L^2$ such that $H\psi\in L^2$.
Then, we describe the domains of the self-adjoint extensions by means of boundary conditions at the origin.

Although we consider a specific class of potentials, still a complete description of the phenomena was not available.
In fact, Arai in \cite{arai1983essential}, analysed potentials as in \eqref{eq:def.V} and he connected the problem of self-adjointess to the quantity $\delta$ defined in \eqref{eq:delta_k}. 
But still, he could only analyse the cases in which $\delta>0$ for any $k>0$:
we do not add any restriction on $\delta$.

In this context the case $\delta>0$ is sub-critical, while it is 
critical if $\delta=0$ for some $k$ and supercritical if $\delta< 0$ for some $k$.
This formulation of criticality is different from the one in \cite{kato1983holomorphic,  adv2013self, arrizabalaga2011distinguished}
but it appears to be suited to this problem, where a particular structure of $\V$ is assumed.
In fact, in the particular case that $\lambda=\nu=0$ and $\V=\frac{\mu}{|x|}\beta$ for all $\mu\in\R$,
the operator
$H$ is essentially self-adjoint on $C^\infty_c(\Rt)^4$ and self-adjoint on $\D(H_0)=H^1(\Rt)^4$, 
see \Cref{cor:lorentz}.

Finally we focus on the distinguished self-adjoint extension: we give a precise description of the domain of the
distinguished self-adjoint extension for $H$ in the sub-critical and critical cases.
In the sub-critical case our results refine the known theory: 
Schmincke's condition (see \ref{eq:schminke}) selects a self-adjoint extension
and we prove that the functions in its domain fulfil
an improved integrability condition.
Moreover, from the algebra of the problem we select a suitable linear combination of both components of the spinor: 
we show that the distinguished self-adjoint extension can be characterized by the fact that this linear combination belongs to $H^1$
 (see \Cref{thm:distinguished.gamma<12,})
and we extend continuously this condition to the critical case for $(\nu,\mu)\neq 0$ in \eqref{eq:def.V}
(see \Cref{thm:distinguished.gamma=0}).
With this definition and in the case of Coulomb potentials, we will show that distinguished
self-adjoint extension is the unique one that has no logarithmic decay at the origin and so it coincides with the self-adjoint extension defined by Esteban and Loss in in \cite{estebanloss}, see \Cref{rem:dist.no.log}.
In the critical case and for $\nu=\mu=0$ we can not define the distinguished self-adjoint extension: in this very particular case a coherent definition of distinguished self-adjoint extension can not be given, see \Cref{rem:dist.anomalous.magn.critical}.

In order to state our results we need to introduce some notations and well known results.
It is well-known  that the free Dirac operator $H_0$ is essentially self-adjoint on $C^\infty_c(\Rt)^4$ and self-adjoint on $\D (H_0):=H^1(\Rt)^4$, see \cite[Theorem 1.1]{thaller}.
We define the \emph{maximal operator} $H_{max}$ as follows:
\begin{equation}\label{eq:defn.maximal.operator} 
    \mathcal{D}(H_{max}):=\seq{\psi\in L^2(\Rt)^4: H\psi\in L^2(\Rt)^4}, 
    \quad
    H_{max} \psi :=H \psi \quad \text{for }\psi \in \mathcal{D}(H_{max}),
\end{equation}
where $H \psi\in L^2(\Rt)^4$ has to be read in the distributional sense:
the linear form 
$\ell_\psi: \varphi\in C^\infty_c(\Rt)^4 \mapsto \int_{\Rt}\psi\,\overline{H\varphi}\,dx$ 
admits a unique extension $\hat{\ell}_\psi$ defined on $L^2(\Rt)^4$ and
by the Riesz Theorem there exists a unique $H_{max}\psi:=\eta\in L^2(\Rt)^4$, 
such that $\hat{\ell}_{\psi}(\cdot)=\langle\eta,\cdot\rangle_{L^2}$.
From \eqref{eq:classicalhardy} it follows that
\begin{equation}\label{eq:liberocontmax}
\D(H_0)\subset\D (H_{max}).
\end{equation}
We define the \emph{minimal operator} $H_{min}$ as follows:
\begin{equation}\label{eq:def.minimal.operator} 
    \mathcal{D}(\Hm):=C^\infty_c(\Rt), 
    \quad 
    H_{min}\psi :=H \psi, \quad \text{for }\psi \in \mathcal{D}(H_{min}).
\end{equation}
It is easy to see that $H_{min}$ is symmetric and $(H_{min})^*=H_{max}$.
Finally, we define $\mathring{H}_{min}$ as follows:
\begin{equation*}\label{eq:def.minimal.operator.o} 
    \mathcal{D}(\mathring{H}_{min}):=C^\infty_c(\Rt\setminus \{0\}), 
    \quad 
    \mathring H_{min}\psi :=H_{min} \psi, \quad \text{for }\psi \in \mathcal{D}(\mathring H_{min}).
\end{equation*}
The operator $\mathring H_{min}$ is symmetric and, for all
  $\psi \in \mathcal{D}(\mathring H_{min})$, $\mathring H_{min}\psi$
  is evaluated in the classical sense.  We remark that
  $\overline{H_{min}}=\overline{\mathring H_{min}}$ (see 
  \cite[Remark 1.1]{arai1983essential}): in particular $({H_{min}})^*=({\mathring H_{min}})^*=H_{max}$.

In this paper we describe self-adjoint extensions $T$ of the minimal operator $H_{min}$.
$T$ is consequently a restriction of the maximal operator, i.e.~
\begin{equation*}
  H_{min} \subseteq T=T^*  \subseteq  H_{max}.
\end{equation*}
In fact the main focus of this paper is studying in detail the restrictions of the maximal operator $H_{max}$.
Following this program, we understand the behaviour of $T$
on the so called \emph{partial wave subspaces} associated to the Dirac equation:
such spaces are left inviariant by $H_0$ and potentials $\mathbb V$ in the class considered in \eqref{eq:def.V}.
We sketch here this topic, referring to \cite[Section 4.6]{thaller} for further details.

Let $Y^l_n$ be the spherical harmonics. They are defined for $n = 0, 1, 2, \dots$, and $l =-n,-n + 1,\dots , n,$ and they satisfy $\Delta_{\mathbb{S}^2} Y^l_n= n(n + 1)Y^l_n$, where $\Delta_{\mathbb{S}^2}$ denotes the usual spherical Laplacian. Moreover, $Y^l_n$ form a complete orthonormal set in $L^2(\mathbb{S}^2)$.
For $j = 1/2, 3/2, 5/2, \dots , $ and $m_j = -j,-j + 1, \dots , j$, set
\begin{align*}
\psi^{m_j}_{j-1/2}&:=
\frac{1}{\sqrt{2j}}
\left(\begin{array}{c}
\sqrt{j+m_j}\,Y^{m_j-1/2}_{j-1/2}\\
\sqrt{j-m_j}\,Y^{m_j+1/2}_{j-1/2}\\
\end{array}\right),
\\
\psi^{m_j}_{j+1/2}&:=\frac{1}{\sqrt{2j+2}}
\left(\begin{array}{c}
\sqrt{j+1-m_j}\,Y^{m_j-1/2}_{j+1/2}\\
-\sqrt{j+1+m_j}\,Y^{m_j+1/2}_{j+1/2}\\
\end{array}\right);
\end{align*}
then  $\psi^{m_j}_{j\pm1/2}$ form a complete orthonormal set in $L^2(\mathbb{S}^2)^2$.
Moreover, we set 
\begin{equation*}
r=|x|,\quad\hat x = x / |x|\quad\text{and}\quad
L=-ix\times\nabla\quad\text{for }x \in \Rt\setminus\{0\}.
\end{equation*}
Then
\begin{equation*}
(\sigma\cdot\hx)\psi^{m_j}_{j\pm 1/2}=\psi^{m_j}_{j\mp1/2},
\quad\text{and}\quad
(1+\sigma\cdot L)\psi^{m_j}_{j\pm 1/2}=\pm(j+1/2)\psi^{m_j}_{j\pm 1/2},
\end{equation*}
where $\sigma=(\sigma_1,\sigma_2,\sigma_3)$ is the vector of \textit{Pauli's matrices}.
For $k_j:=\pm(j+1/2)$ we set
\begin{equation*}
\Phi^+_{m_j,\pm(j+1/2)}:=
\left(\begin{array}{c}
i\,\psi^{m_j}_{j\pm1/2}\\
0
\end{array}\right),
\quad
\Phi^-_{m_j,\pm(j+1/2)}:=
\left(\begin{array}{c}
0\\
\psi^{m_j}_{j\mp1/2}
\end{array}\right).
\end{equation*}
Then, the set $\{\Phi^+_{m_j,k_j},\Phi^-_{m_j,k_j}\}_{j,k_j,m_j}$ is a 
complete orthonormal basis of $L^2(\mathbb{S}^2)^4$. 

We define now the following space:
\begin{equation*}
\mathcal{H}_{m_j,k_j}:=
\seq*{ 
\frac{1}{r}\left(
f^+_{m_j,k_j}(r)\Phi^+_{m_j,k_j}(\hx)+
f^-_{m_j,k_j}(r)\Phi^-_{m_j,k_j}(\hx)\right)
\in L^2(\Rt)
\mid
f^\pm_{m_j,k_j}\in L^2(0,+\infty)}.
\end{equation*}
From \cite[Theorem 4.14]{thaller} we know that the operators $H_0$, $\mathring H_{min}$ and $H_{max}$ leave the partial wave subspace $\H_{m_j,k_j}$ invariant and their action can be decomposed in terms of the basis $\seq{\Phi^+_{m_j,k_j},\Phi^-_{m_j,k_j}}$ as follows:
\begin{align*}
H_0&\cong
\bigoplus_{j=\frac{1}{2},\frac{3}{2},\dots}^\infty\,\bigoplus_{m_j=-j}^j\,
\bigoplus_{k_j=\pm(j+1/2)}\,
h^0_{m_j,k_j},
\\
\mathring H&\cong
\bigoplus_{j=\frac{1}{2},\frac{3}{2},\dots}^\infty\,\bigoplus_{m_j=-j}^j\,
\bigoplus_{k_j=\pm(j+1/2)}\,
h_{m_j,k_j},
\\
 H^*&\cong
\bigoplus_{j=\frac{1}{2},\frac{3}{2},\dots}^\infty\,\bigoplus_{m_j=-j}^j\,
\bigoplus_{k_j=\pm(j+1/2)}\,
{h}^*_{m_j,k_j},
\end{align*}
where ``$\cong$'' means that the operators are unitarily equivalent, and the action of $H_0$ is represented by
\begin{equation*}
\begin{split}
{D}( h^0_{m_j,k_j})&= 
 \seq*{(f^+,f^-)\in L^2(0,+\infty)^2: \ 
\left(
\de_r\pm \frac{k_j}{r}
\right)f^\pm
\in 
L^2(0,+\infty)},\\
h^0_{m_j,k_j}
(
  f^+, f^-)&:
=
\begin{pmatrix}
m & -\partial_r+\frac{k_j}{r}\\
\partial_r+\frac{k_j}{r} & -m
\end{pmatrix}
\begin{pmatrix}
  f^+ \\
 f^-
\end{pmatrix},
\end{split}
\end{equation*}
the action of $\mathring H_{min}$ is represented by
\begin{equation}\label{eq:dirac.spherical}
{D}(  h_{m_j,k_j})= C^\infty_c(0,+\infty)^2,
\quad
h_{m_j,k_j}
( f^+,  f^-):
=\left(
\begin{array}{cc}
m+\frac{\nu+\mu}{r} & -\partial_r+\frac{k_j+\lambda}{r}\\
\partial_r+\frac{k_j+\lambda}{r} & -m+\frac{\nu-\mu}{r}
\end{array}\right)
\begin{pmatrix}
  f^+ \\
 f^-
\end{pmatrix}.
\end{equation}
and the action of $H_{max}$ is represented by
\begin{equation}\label{eq:dirac.spherical*}
\begin{split}
{D}( h^*_{m_j,k_j})&= \seq{(f^+,f^-)\in L^2(0,+\infty): h^*_{m_j,k_j}(f^+,f^-) \in L^2(0,+\infty)^2},\\
h^*_{m_j,k_j}
(
  f^+, f^-)&:
=\left(
\begin{array}{cc}
m+\frac{\nu+\mu}{r} & -\partial_r+\frac{k_j+\lambda}{r}\\
\partial_r+\frac{k_j+\lambda}{r} & -m+\frac{\nu-\mu}{r}
\end{array}\right)
\begin{pmatrix}
  f^+ \\
 f^-
\end{pmatrix}.
\end{split}
\end{equation}
where $h^*_{m_j,k_j}(f^+,f^-)$ has to be read in the distributional sense as done in \eqref{eq:defn.maximal.operator}.

Finally, by construction, $h^*_{m_j,k_j}$ is the adjoint of $h_{m_j,k_j}$.

In this framework the operator $T$ can be decomposed as
\begin{equation*}
    T \cong \bigoplus_{j=\frac{1}{2},\frac{3}{2},\dots}^\infty\,
    \bigoplus_{m_j=-j}^j\, \bigoplus_{k_j=\pm(j+1/2)}\,
    t_{m_j,k_j}.
\end{equation*}
We will characterize all the self-adjoint operators $T$ such that every $t_{m_j,k_j}$ is 
sef-adjoint: this property is linked to the quantity
\begin{equation}
  \label{eq:defn.delta}
  \delta=\delta(k_j,\lambda, \mu, \nu):=(k_j+\lambda)^2+\mu^2-\nu^2.
\end{equation}
We can now state the following theorems, main results of this paper.
In these we fix $j\in \seq{1/2,3/2,\dots}, m_j\in\seq{-j,\dots,j}$, $k_j\in\seq{j+1/2,-j-1/2}$
and let $\delta$ as in \eqref{eq:defn.delta}.

\begin{theorem}[Case $\delta \geq 1/4$]
\label{thm:Hk.good}
Let $h_{m_j,k_j}$ and $h^*_{m_j,k_j}$ be defined respectively as in \eqref{eq:dirac.spherical} and \eqref{eq:dirac.spherical*}
for $\nu,\mu,\lambda\in \R$, and $\delta\in\R$ as in \eqref{eq:defn.delta}.
Assume $\delta \geq \frac14$ and set $\gamma:=\sqrt{\delta}$.
The following hold:
  \begin{enumerate}[label=(\roman*)]
  \item\label{thm:Hk.gamma>12} If $\gamma > \frac12$ then $h_{m_j,k_j}$ is essentially self-adjoint on  
    $C_c^{\infty}(0,+\infty)^2$
    and 
	\begin{equation*}
    \D(\overline{h_{m_j,k_j}})=\D(h^0_{m_j,k_j}).
    \end{equation*}
      \item\label{thm:Hk.gamma=12} If $\gamma= \frac12$ then $h_{m_j,k_j}$ is essentially self adjoint on 
    $C_c^{\infty}(0,+\infty)^2$ and 
    \begin{equation*}
    \mathcal{D}(h^0_{m_j,k_j})\subset \D(\overline{h_{m_j,k_j}}).
    \end{equation*}
  \end{enumerate}
\end{theorem}
\begin{theorem}[Case $0 \leq \delta < 1/4$]
\label{thm:Hk.subandcritical}
Under the same assumptions of \Cref{thm:Hk.good}, assume 
$0 \leq \delta < \frac14$ and set $\gamma:=\sqrt{{\delta}}$.
The following hold:
  \begin{enumerate}[label=(\roman*)]
  \item\label{thm:Hk.gamma<12} If $0<\gamma < 1/2$ there is a one (real) parameter family 
    $\left(t(\theta)_{m_j,k_j}\right)_{\theta \in [0,\pi)}$
    of self-adjoint extensions  $h_{m_j,k_j}\subset t(\theta)_{m_j,k_j}=t(\theta)_{m_j,k_j}^* \subset h^*_{m_j,k_j}$.
Moreover $(f^+, f^-)\in
   \mathcal{D}\left(t(\theta)_{m_j,k_j}\right)$ if and only if
    $(f^+, f^-)\in
   \mathcal{D}(h^*_{m_j,k_j})$  and
    there exists $(A^+,A^-)\in \C^2$ such that     
\begin{equation}\label{eq:defn.Hk.<1/2}
\begin{split}
A^+\sin \theta  &+ A^-\cos \theta =0,\\
                \lim\limits_{r\to 0}
         \Bigg|
          \begin{pmatrix} 
          f^+(r) \\ 
          f^-(r) 
          \end{pmatrix} 
          &- D 
          \begin{pmatrix}
           A^+ r^{\gamma} \\
            A^- r^{-\gamma}
          \end{pmatrix}\Bigg|
          r^{-1/2}
          =0,
          \end{split}
          \end{equation}
where $D\in \R^{2\times 2}$ is invertible and
\begin{equation}
  \label{eq:defn.D}
  D:= 
  \begin{cases}	
 \frac{1}{2\gamma(\lambda + k - \gamma)}
  \begin{pmatrix}
    \lambda + k_j - \gamma & \nu-\mu \\
     -(\nu + \mu)        & -(\lambda + k_j - \gamma)
  \end{pmatrix}
  \quad &\text{ if }\lambda + k_j - \gamma \neq 0,\\
    \frac{1}{-4\gamma^2}
    \begin{pmatrix}
      \mu - \nu            & 2\gamma \\
     2\gamma & -(\nu + \mu) 
    \end{pmatrix}
    \quad &\text{ if }\lambda + k_j - \gamma = 0.
\end{cases}
\end{equation}
Conversely, any self-adjoint extension $t_{m_j,k_j}$ of $h_{m_j,k_j}$ verifies $t_{m_j,k_j}=t(\theta)_{m_j,k_j}$ for some $\theta \in [0,\pi)$.

\item\label{thm:Hk.gamma=0} If $\gamma=0$ 
 there is a one (real) parameter family 
    $\left(t(\theta)_{m_j,k_j}\right)_{\theta \in [0,\pi)}$
    of self-adjoint extensions  $h_{m_j,k_j}\subset t_{m_j,k_j}(\theta)=t_{m_j,k_j}(\theta)^* \subset {h^*}_{m_j,k_j}$. 
    Moreover $(f^+, f^-)\in
   \mathcal{D}\left(t(\theta)_{m_j,k_j}\right)$ if and only if 
       $(f^+, f^-)\in
   \mathcal{D}(h^*_{m_j,k_j})$  and
    there exists $(A^+,A^-) \in \C^2$ such that  
    \begin{equation}    \label{eq:defn.Hk.=0}
    \begin{split}
    A^+\sin \theta  &+ A^-\cos \theta =0,\\    
              \lim\limits_{r\to 0} 
          \Bigg|\begin{pmatrix} f^+(r) \\ f^-(r) \end{pmatrix} 
          &- (M \log r + \mathbb{I}_2) 
          \begin{pmatrix}
            A^+ \\
            A^-
          \end{pmatrix}\Bigg|r^{-1/2}
          =0,
          \end{split}
    \end{equation}
with $M\in\R^{2\times 2}$, $M^2=0$ to be
\begin{equation}\label{eq:defn.M}
   M:= 
\begin{pmatrix}
  -(k_j+\lambda) & - \nu + \mu \\
  \nu + \mu           & k_j+\lambda
\end{pmatrix}.
 \end{equation}
Conversely, any self-adjoint extension $t_{m_j,k_j}$ of $h_{m_j,k_j}$ verifies $t_{m_j,k_j}=t(\theta)_{m_j,k_j}$ for some $\theta \in [0,\pi)$.
\end{enumerate}
\end{theorem}

\begin{theorem}[Case $ \delta< 0$]
  \label{thm:Hk.supercritical}
Under the same assumptions of \Cref{thm:Hk.good}, assume 
$\delta < 0$ and set $\gamma:=\sqrt{\abs{\delta}}$.
there is a one (real) parameter family 
    $\left(t(\theta)_{m_j,k_j}\right)_{\theta \in [0,\pi)}$
    of self adjoint extensions  $h_{m_j,k_j}\subset t(\theta)_{m_j,k_j}=t(\theta)_{m_j,k_j}^* \subset h^*_{m_j,k_j}$.
 Moreover $(f^+, f^-)\in
   \mathcal{D}\left(t(\theta)_{m_j,k_j}\right)$ if and only if 
       $(f^+, f^-)\in
   \mathcal{D}(h^*_{m_j,k_j})$  and
    there exists $A \in \C$ such that 
    \begin{equation}\label{eq:defn.A.<0}
 \lim\limits_{r\to 0}
      \Bigg|\begin{pmatrix} f^+(r) \\ f^-(r) \end{pmatrix}
          -  D 
          \begin{pmatrix}
            A e^{i\theta} r^{i\gamma}  \\
            A \sqrt{\frac{\nu+\mu}{\nu-\mu}}e^{-i\theta}  r^{-i\gamma} 
          \end{pmatrix}\Bigg|r^{-1/2}
          =0,
      \end{equation}
where $D \in \C^{2\times 2}$ is invertible and equals
\begin{equation}
  \label{eq:defn.D.negative}
  D:= 
   \frac{1}{2i \gamma(\lambda + k - i\gamma)}
  \begin{pmatrix}
    \lambda + k - i\gamma & \nu-\mu \\
     -(\nu + \mu)        & -(\lambda + k - i\gamma)
  \end{pmatrix}.
\end{equation}
Conversely, any self-adjoint extension $t_{m_j,k_j}$ of $h_{m_j,k_j}$ verifies $t_{m_j,k_j}=t(\theta)_{m_j,k_j}$ for some ${\theta \in [0,\pi)}$.
\end{theorem}

\begin{remark}
The quantity $\delta$ in \eqref{eq:defn.delta}  was already considered
in \cite{arai1983essential}: in Theorem 2.7 Arai studies properties of self-adjointness for the restriction of $T$ 
on the partial wave subspaces for $\delta >0$, by means of the Von Neumann deficiency indices theory.
We can treat the general case $\delta \in \R$, and our approach 
has the value of giving more informations on the domain of self-adjointness.
\end{remark}
\begin{remark}
In the proof of Theorems \ref{thm:Hk.good}, \ref{thm:Hk.subandcritical}
and \ref{thm:Hk.supercritical} we rely on the properties of $\V$
\begin{align}\label{eq:K.commute.V}
     [\mathbf K, {\V}(x)]&=0,\\
	\label{eq:d_r.commute.rV}    
     [\partial_r, \abs{x}\V(x)]&=0,
\end{align}
where
$\mathbf K$ is the \emph{spin-orbit operator} defined as
\[
\mathbf K:=
\begin{pmatrix}
(1+\sigma\cdot L)& 0\\
0&-(1+\sigma\cdot L)
\end{pmatrix}.
\]
Indeed from \eqref{eq:K.commute.V} we have that $\V$ leaves the partial wave subspaces $\H_{m_j,k_j}$ invariant and from \eqref{eq:d_r.commute.rV} we have that ${\V}$ is critical with respect to the scaling associated to the gradient.
 A general potential of this kind is represented in the basis of $\H_{m_j,k_j}$ by the complex hermitian matrix
\begin{equation}  
\frac1{\abs{x}}
\begin{pmatrix}
\nu + \mu & \lambda - i \xi \\
\lambda + i \xi & \nu - \mu 
\end{pmatrix},
\end{equation}
with $\nu,\mu,\lambda,\xi \in \R$.
Such a matrix describes the potential
\begin{equation}\label{eq:Vgeneric}
 \frac{1}{|x|}\left( \nu \mathbb{I}_4 +\mu \beta +
  \lambda \left(-i\alpha\cdot\frac{x}{\abs{x}}\,\beta \right)
  +
  \xi \frac{\alpha \cdot x}{\abs{x}}  
  \right),
\end{equation}
thanks to the fact that, with respect to the basis $\{\Phi^+_{m_j,k_j},\Phi^-_{m_j,k_j}\}_{j,k_j,m_j}$,
\begin{equation*}
\beta \cong 
\begin{pmatrix}
1 & 0 \\
0 & -1
\end{pmatrix}, 
\quad 
- i \alpha \cdot \frac{x}{\abs{x}} \cong
\begin{pmatrix}
0 & -1 \\
1 & 0
\end{pmatrix}.
\end{equation*}
The term in $\xi$ in \eqref{eq:Vgeneric} can be removed with a change of gauge, since
$e^{i \xi \log \abs{x} }\xi \alpha\cdot x \abs{x}^{-2} =  -i\alpha\cdot \nabla e^{i \xi \log \abs{x} }$.
For these reasons we are considering potentials as in   \eqref{eq:def.V} 
in our results, and they are the most general potentials that can be treated with this approach.
This rigidity is not essential, since the self-adjointness is stable under $L^\infty$ perturbations:
 for potentials $\mathbb W(x)$ such that $\mathbb W - \V \in L^\infty(\R^3;\C^{4 \times 4})$,
$H + \mathbb W (x) $ is self-adjoint whenever $H + \mathbb{V}(x)$ is self-adjoint.
In detail, if $\mathbb W (x) = w(x)/\abs{x}$, this amounts to require that for almost all $x \in \R^3$
\begin{equation*}
  \abs*{w(x) - \left(\nu \mathbb{I}_4 +\mu \beta -i \lambda \alpha\cdot\frac{x}{\abs{x}}\right)} \leq C \abs{x},
\end{equation*}
for some $\lambda,\mu,\nu \in \R$ and $C>0$.
More general perturbation results are possible, for example exploting the Kato-Rellich perturbation theory, and they will be matter of future investigation.
\end{remark}

\begin{corollary}[Lorentz-scalar Potential]\label{cor:lorentz}
Let $\mathbb V$, $T_{max}$ and $T_{min}$ be defined as in 
\eqref{eq:def.V},
\eqref{eq:defn.maximal.operator},
\eqref{eq:def.minimal.operator} respectively, with $\lambda=\nu=0$.
Then then for all $\mu\in\R$, $T_{min}$ is essentially self-adjoint on 
    $C_c^{\infty}(\Rt\setminus\seq{0})^4$,
    and $\mathcal{D}(\overline{T_{min}})=\mathcal{D}(T_{max})= H^1(\Rt)^4$.
\end{corollary}

In the case $0<\delta<1/4$, the \emph{distinguished} self-adjoint extension is of particular interest among the self-adjoint extensions given in \Cref{thm:Hk.subandcritical}. 
We need the following notation: for $a\in\R$ set 
\begin{equation*}
\begin{split}
&\mathcal{D}(r^{-a}, \R^3):= \{\psi \in L^2 (\R^3) : \abs{x}^{-a}\psi \in L^2(\R^3)\}, \\
&\mathcal{D}(r^{-a}, (0,+\infty) ):= \{f \in L^2 (0,+\infty) : r^{-a}f \in L^2(0,+\infty)\},
\end{split}
\end{equation*}
and, for 
\begin{equation*}
\psi(x)=
  \sum_{j,k_j,m_j} 
\frac{1}{r}\left(
f^+_{m_j,k_j}(r)\Phi^+_{m_j,k_j}(\hx)+
f^-_{m_j,k_j}(r)\Phi^-_{m_j,k_j}(\hx)\right),
\end{equation*}
it is true that $\psi \in \mathcal{D}(r^{-a}, \R^3)^4 $ if and only if $f^+_{m_j,k_j}, f^-_{m_j,k_j}\in \mathcal{D}(r^{-a}, (0,+\infty) )$ 
for all $j,m_j,k_j$. In the following we will simply write $\mathcal{D}(r^{-a})$, since it will be clear from the context to which set we are referring.

In the literature,  the {distinguished self-adjoint extension} is defined as the unique one whose domain is contained in $\mathcal{D}(r^{-1/2})$ (among other definitions, see \cite{gallone2017self}), but this definition is no longer valid in the critical case, since no extension verifies such a property. 
From a more physical perspective, such extension is characterized by the fact that a space of regular functions is dense (in some sense) in its domain. 
In this context, we deduce a norm associated to our problem and we characterize the distinguished extension by the following fact:
regular functions approximate in such norm a particular linear combination (deduced by the particular algebra of the problem) of both components of the spinor.
Then we extended this definition to the critical case. Nevertheless this definition does not work in the particular case of $\mathbb{V}$ defined in \eqref{eq:def.V} being purely anomalous magnetic (i.~e.~ $\nu=\mu=0$). 
In this case we deduce a notion of distinguished extension cannot be given.

This motivates the following propositions, where we collect  properties of the distinguished self-adjoint extension in the case $0\leq \delta<1/4$.

\begin{proposition}[Distinguished Self-Adjoint Extension for the subcritical case]
  \label{thm:distinguished.gamma<12}
Let $0<\gamma<\frac12$ and
  $\left(t(\theta)_{m_j,k_j}\right)_{\theta \in [0,\pi)}$ be the one (real) parameter family 
  of self adjoint extensions  considered
  in \ref{thm:Hk.gamma<12} of \Cref{thm:Hk.subandcritical}. 
  
  Then the following are equivalent:
  \begin{enumerate}[label=(\roman*)]
  \item \label{lab:distinguished.gamma<12ii}$\theta=0$; 
    \item \label{lab:distinguished.gamma<12iv} $\mathcal{D}\left(t(\theta)_{m_j,k_j}\right) \subseteq \mathcal{D}(r^{-1/2})^2$;
      \item \label{lab:distinguished.gamma<12iii}$\mathcal{D}\left(t(\theta)_{m_j,k_j}\right) \subseteq
   \mathcal{D}(r^{-a})^2$ with $a\in\left[\frac{1}{2},\frac{1}{2}+\gamma\right)$;
     \item\label{lab:distinguished.gamma<120}
  for any 
  $(f^+_{m_j,k_j},f^-_{m_j,k_j})\in \mathcal{D}\left(t(\theta)_{m_j,k_j}\right)$, setting
   \begin{equation}\label{eq:def.phi^-.gamma<12}
\varphi^-_{m_j,k_j}:=
\begin{cases}
(\nu+\mu)f^+_{m_j,k_j}+(k_j+\lambda-\gamma)f^-_{m_j,k_j} &\text{if}\ k_j+\lambda-\gamma\neq 0,\\
-2\gamma f^+_{m_j,k_j}+(-\nu+\mu)f^-_{m_j,k_j}&\text{if}\ k_j+\lambda-\gamma= 0,
\end{cases}
 \end{equation}
 we have $\varphi^-_{m_j,k_j}\in \mathcal{J}= \seq*{u\in AC[0,M]\ \text{for any}\ M>0:\ 
 \frac{u}{r},\ u' \in L^2(0,+\infty)}$.
  \end{enumerate}
 \end{proposition}

\begin{proposition}[Distinguished Self-Adjoint Extension for the critical case]
  \label{thm:distinguished.gamma=0} 
Let $\gamma=0$ and assume that in \eqref{eq:def.V} $(\nu,\mu)\neq (0,0)$. Let
  $\left(t(\theta)_{m_j,k_j}\right)_{\theta \in [0,\pi)}$ be the one (real) parameter family 
  of self adjoint extensions  considered
  in \ref{thm:Hk.gamma=0} of \Cref{thm:Hk.subandcritical}.
 Then the following are equivalent:
  \begin{enumerate}[label=(\roman*)]
  \item\label{lab:distinguished.gamma=0i0}
   for any 
  $(f^+_{m_j,k_j},f^-_{m_j,k_j})\in \mathcal{D}\left(t(\theta)_{m_j,k_j}\right)$, setting
   \begin{equation}\label{eq:def.phi^-.gamma=0}
\varphi^-_{m_j,k_j}:=
\begin{cases}
(\nu+\mu)f^+_{m_j,k_j}+(k_j+\lambda) f^-_{m_j,k_j} &\text{if}\ \nu+\mu\neq 0,\\
-2\nu f^-_{m_j,k_j} &\text{if}\ \nu+\mu= 0.
\end{cases}
 \end{equation}
  we have $\varphi^-_{m_j,k_j}\in \mathcal{J}=\seq*{u\in AC[0,M]\ \text{for any}\ M>0:  \frac{u}{r},\ u' \in L^2(0,+\infty)}$;
  \item \label{lab:distinguished.gamma=0iii} $\theta=
\begin{cases}
\operatorname{arccot}\left(\frac{k_j+\lambda}{\nu+\mu}\right)&\text{if} \ \nu+ \mu\neq 0,\\
0&\text{if}\ \nu+\mu=0.
 \end{cases} 
$
 \end{enumerate}
\end{proposition}
\begin{remark}
  The space $\mathcal{J}$ is the completion of $C^\infty_c(0,+\infty)$ with respect 
  to the norm 
  \begin{equation}
    \label{eq:3}
    \norm{\varphi}_{\mathcal{J}}:=
    \left(\int_0^\infty  \abs{\partial_r(r^a \varphi(r))}^2 r^{-2a} dr\right)^{1/2},\quad\text{for}\ a\neq -1/2.
  \end{equation}
  Such a norm arises naturally from the study of the operator $H$.
  We prove this density and give more details about the space $\mathcal{J}$ in \Cref{sec:distinguished}.
\end{remark}
\begin{remark}\label{rem:dist.no.log}
Under the assumptions of \Cref{thm:distinguished.gamma=0}, 
from \eqref{eq:defn.Hk.=0} we get that, among all the self-adjoint extensions in the family $\left( t(\theta)_{m_j,k_j}\right)_{\theta\in[0,\pi)}$ described by \Cref{thm:distinguished.gamma=0}, there is a unique one that has no logarithmic decay at the origin.
Indeed, this is a consequence of the fact that the kernel of the matrix $M$ defined in \eqref{eq:defn.M} has complex dimension one. 
Thanks to \eqref{eq:defn.Hk.=0} we deduce that the unique self-adjoint extension that has no logarithmic decay at the origin is the distinguished one described in \Cref{thm:distinguished.gamma=0}.
In particular,
in the case $\mathbb V(x) =- \nu/\abs{x} $, when $\nu \in (\sqrt{3}/2,1]$, the extensions
\begin{gather*}
T(0,0,0,0) \cong
    \left(
   t(0)_{\frac{1}{2},1}
   \oplus
   t(0)_{-\frac{1}{2},1}
   \oplus
  t(0)_{\frac{1}{2},-1}
   \oplus
   t(0)_{-\frac{1}{2},-1}
    \right)
\oplus \left( 
\bigoplus_{\substack{j,k_j,m_j \\ \abs{k_j} > 1}}
 h^*_{m_j,k_j}
 \right),\quad  \sqrt{3}/2< \nu<1,
 \\
T\left(\tfrac{3\pi}{4},\tfrac{3\pi}{4},\tfrac{\pi}{4},\tfrac{\pi}{4}\right)
\cong
\left(
   t\left({\small \tfrac{3\pi}{4}}\right)_{\frac{1}{2},1}
   \oplus
    t\left(\tfrac{3\pi}{4}\right)_{-\frac{1}{2},1}
   \oplus
   t\left(\tfrac{\pi}{4}\right)_{\frac{1}{2},-1}
   \oplus
    t\left(\tfrac{\pi}{4}\right)_{-\frac{1}{2},-1}
    \right)
\oplus \left( 
\bigoplus_{\substack{j,k_j,m_j \\ \abs{k_j} > 1}}
 h^*_{m_j,k_j}
 \right),\quad \nu=1,
\end{gather*}
coincide with the ones considered in \cite[Section 1.5]{esteban2017domains}.
\end{remark}

\begin{remark}\label{rem:dist.critical.nu}
  For $\nu \in (0,1]$ and $a:=m\sqrt{1-\nu^2}\in[0,m)$, the function
  \begin{equation*}
    \psi_a(x) =    
    \frac{e^{-\sqrt{m^2-a^2}\abs{x}}}{\abs{x}^{1-a/m}}
    \begin{pmatrix}
            1 \\ 1 \\
      i \sqrt{\frac{m-a}{m+a}}  \sigma \cdot \hx \cdot
      \begin{pmatrix}
        1 \\ 1
      \end{pmatrix}
      \\
    \end{pmatrix}
  \end{equation*}
is solution to the equation 
\begin{equation*}
  \left(-i \alpha\cdot \nabla + m\beta - \frac{\nu}{\abs{x}} \right)\psi = a \psi,
\end{equation*}
i.e.~$\psi_a$ is an eigenfunction for the Dirac-Coulomb operator of eigenvalue $a$. 
Remembering that 
\begin{align*}
\Phi^+_{\frac12,-1}&= 
  \frac1{\sqrt{4\pi}}
  \begin{pmatrix}
    i
    \\
    0
    \\
    0
    \\
    0
  \end{pmatrix}, 
  &
  \Phi^+_{-\frac12,-1}&= 
  \frac1{\sqrt{4\pi}}
  \begin{pmatrix}
    0
    \\
    i
    \\
    0
    \\
    0
  \end{pmatrix}, 
  \\
  \Phi^-_{\frac12,-1}&= 
  \frac1{\sqrt{4\pi}}
  \begin{pmatrix}
    0
    \\
    0
    \\
    \sigma \cdot \hx \cdot 
    \begin{pmatrix}
      0 \\ 1
    \end{pmatrix}
  \end{pmatrix}, 
  &
  \Phi^-_{-\frac12,-1}&= 
  \frac1{\sqrt{4\pi}}
  \begin{pmatrix}
    0
    \\
    0
    \\
    \sigma \cdot \hx \cdot 
    \begin{pmatrix}
      0 \\ 1
    \end{pmatrix}
  \end{pmatrix}, 
  \end{align*}
it is easy to show that in the sub-critical case, that is $\nu \in (\sqrt{3}/2,1)$, $\psi_a\in
\D(T(0,0,0,0))$ and in the critical case, that is $\nu=1$,  $\psi_0\in\mathcal{D}\left(T\left(\frac{3\pi}{4},\frac{3\pi}{4},\frac{\pi}{4},\frac{\pi}{4}\right)\right)$,  
thanks to the explicit characterization of these
domains given by \Cref{thm:Hk.subandcritical}.
\end{remark}

\begin{remark}[Distinguished self-adjoint extension for the critical anomalous magnetic potential]
\label{rem:dist.anomalous.magn.critical}
Assuming that  $(\nu,\mu)= (0,0)$ in \eqref{eq:def.V} 
it is not possible to give a coherent definition of distinguished self-adjoint extension in the critical case.
Indeed, under this hypothesis, $\gamma=|k_j+\lambda|$; let $\gamma=0$ and let $\left(t(\theta)_{m_j,k_j}\right)_{\theta\in [0,\pi)}$ be the one-parameter  family of self-adjoint extension described in \emph{\ref{item:properties.domain.gamma=0}} in \Cref{thm:Hk.subandcritical}. Then for any $\theta\in [0,\pi)$ and for any $(f^+,f^-)\in\mathcal{D}\left(t(\theta)_{m_j,k_j}\right)$, defining 
$\varphi^-_{m_j,k_j}$ as in \eqref{eq:def.phi^-.gamma=0}, we get that $\varphi^-_{m_j,k_j}=0$.
In other words \emph{\ref{lab:distinguished.gamma=0i0}} of \Cref{thm:distinguished.gamma=0} is verified for any $\theta\in [0,\pi)$, as a consequence of the fact that the matrix $M$ defined in \eqref{eq:defn.M} vanishes.
Thus, from \eqref{eq:defn.Hk.=0} we deduce that for any $\theta\in [0,\pi)$ all functions in $\D\left( t(\theta)_{m_j,k_j}\right)$ do not admit logarithmic decay at zero differently from what happens in the case $(\nu,\mu)\neq (0,0)$, see also \Cref{rem:dist.no.log}.

This incongruence can be observed using a different approach:  in the sub-critical case, we find a spectral condition that characterizes  the distinguished self-adjoint extension and 
we realize that it is not possible to extend continuously this condition to the critical case. Indeed, let $0<\gamma<1/2$ and assume that $\left( t(\theta)_{m_j,k_j} \right)_{\theta\in [0,\pi)}$ is the one-parameter  family of self-adjoint extension defined in \Cref{thm:Hk.subandcritical}.
Let us find eigenvalues for $t(\theta)_{m_j,k_j}$. The $L^2$--solutions of the following equation for $a\in (-m,m)$:
\begin{equation*}\label{eq:cerca.eigen}
\begin{pmatrix}
m+a & -\partial_r+\frac{k_j+\lambda}{r}\\
\partial_r+\frac{k_j+\lambda}{r} & -(m-a)
\end{pmatrix}
\begin{pmatrix}
  f^+ \\
 f^-
\end{pmatrix}=0.
\end{equation*}
are
\begin{equation}\label{eq:eigenfunctions.no.critical}
\begin{split}
f^+(r)&:=
\begin{cases}
A\sqrt{m-a}\,\sqrt{r}K_{\gamma+1/2}\left(\sqrt{m^2-a^2}\, r\right)&\text{if} \ k_j+\lambda>0,\\
A\sqrt{m-a}\,\sqrt{r}K_{\gamma-1/2}\left(\sqrt{m^2-a^2}\, r\right)&\text{if} \ k_j+\lambda<0,
\end{cases}\\
f^-(r)
&:=
\begin{cases}
-A\sqrt{m+a}\,\sqrt{r}K_{\gamma-1/2}\left(\sqrt{m^2-a^2}\, r\right)&\text{if} \ k_j+\lambda>0,\\
-A\sqrt{m+a}\,\sqrt{r}K_{\gamma+1/2}\left(\sqrt{m^2-a^2}\, r\right)&\text{if} \ k_j+\lambda<0,
\end{cases}
\end{split}
\end{equation}
where $K$ is the second-order modified Bessel function and $A\neq 0$.
By \cite[Equation 10.30.2]{olver2010nist}, we get that as $r\to 0$
\begin{align*}
f^+(r)&\sim
\begin{cases}
\tilde{A}\sqrt{m-a}\,r^{-\gamma}&\text{if} \ k_j+\lambda>0,\\
\tilde{A}\sqrt{m-a}\,r^\gamma&\text{if} \ k_j+\lambda<0,
\end{cases}\\
f^-(r)&\sim
\begin{cases}
-\tilde{A}\sqrt{m+a}\,r^{\gamma}&\text{if} \ k_j+\lambda>0,\\
-\tilde{A}\sqrt{m+a}\,r^{-\gamma}&\text{if} \ k_j+\lambda<0.
\end{cases}
\end{align*}
We realize that, for any $a\in (-m,m)$ there exists only one $\theta \in [0,\pi)$ such that $(f^+,f^-)$ defined in \eqref{eq:eigenfunctions.no.critical} belongs to $\mathcal{D}\left(t(\theta)_{m_j,k_k}\right)$.
Such $\theta$ is uniquely determined by the condition 
\[
\begin{cases}
\sin\theta\sqrt{m+a}+\cos\theta\sqrt{m-a}=0&\text{if}\ k_j+\lambda>0,\\
\sin\theta\sqrt{m-a}+\cos\theta\sqrt{m+a}=0&\text{if}\ k_j+\lambda<0.
\end{cases}
\]
Thus, the distinguished self-adjoint extension $t(0)_{m_j,k_j}$ does not have any eigenvalue $a\in (-m,m)$, but it is characterized by the fact that if $k_j+\lambda>0$, it has $m$ as a resonance and if $k_j+\lambda<0$, it has $-m$ as a resonance.
This spectral relation depends on the sign of $k_j+\lambda$ and so it does not have any 
continuous prolongation to the critical case where $k_j+\lambda=0$.
\end{remark}
\subsection*{Acknowledgements}
We wish to thank Luis Vega for addressing us to this interesting and challenging
problem, and for precious discussions and advices.
This research is supported by ERCEA Advanced Grant 2014 669689 - HADE, by the MINECO project MTM2014-53850-P, by Basque Government project IT-641-13 and also by the Basque Government through the BERC 2014-2017 program and by Spanish Ministry of Economy and Competitiveness MINECO: BCAM Severo Ochoa excellence accreditation SEV-2013-0323.
The first author also acknowledges the Istituto Italiano di Alta Matematica ``F.~Severi''.

\section{Trace theorems and Hardy-type inequalities}\label{sec:hardy}
This section is devoted to \emph{Trace theorems} and \emph{Hardy-type} inequalities. These are very useful tools that we will use to prove \Cref{thm:Hk.good},
\ref{thm:Hk.subandcritical}, and \Cref{thm:Hk.supercritical}.
For sake of clarity we prove the following well-known result:
\begin{lemma}\label{lem:fund.thm.calc}
Let $f$ be a distribution on $(a,b)\subset\R$ such that $f'$ is an integrable function on $(a,b)$. Then $f\in AC[a,b]$ and 
\begin{equation}\label{eq:fund.thm.calc}
f(t)-f(s)=\int_s^t f'(r)\ dr \qquad \text{for any}\ s,t\in [a,b].
\end{equation}
\end{lemma}
\begin{proof}
For any $t\in[a,b]$ we set
\begin{equation*}
g(t):=\int_a^t f'(r)\ dr.
\end{equation*}
Thanks to the integrability of $f'$ we get that $g\in AC[a,b]$ and so $g$ is differentiable almost everywhere on $[a,b]$. Then for almost every $t\in [a,b]$
\begin{equation}\label{eq:g'=f'}
g'(t)=\lim_{h\to 0}\frac{g(t+h)-g(t)}{h}=\lim_{h\to 0}\frac{1}{h}\int_{t}^{t+h} f'(r) dr = f'(t),
\end{equation}
where in the last equality we used Lebesgue differentiation Theorem. Thanks to \eqref{eq:g'=f'} there exists $c\in \C$ such that $f=g+c$ in the sense of distributions, that gives $f\in AC[a,b]$ and \eqref{eq:fund.thm.calc}.
\end{proof}
Let us give some trace properties.
\begin{proposition}\label{prop:hardy1d}
Let $f$ be a distribution on $(0,+\infty)$. Let us assume that there exist $a\in\R$ such that 
\begin{equation}\label{eq:ipo.der}
\int_0^{+\infty}|f'(r)|r^{2a}\,dr<\infty .
\end{equation}

Then $f\in AC[\epsilon,M]$ for any $0<\epsilon<M<+\infty$ and the following hold: 
\begin{enumerate}[label=(\roman*)]
\item\label{item:trace.gamma<1/2}
If $a<\frac{1}{2}$, then $f \in AC[0,1]$ and 
\begin{equation}
  \label{eq:trace.0}
  \lim_{t\to 0} \, \abs{f(t)-f(0)}t^{-\left(\frac12 - a\right)} = 0.
\end{equation}
\item\label{item:trace.gamma>1/2}
If $a>\frac{1}{2}$, there exists $f(+\infty)\in \C$ such that
\begin{equation}
  \label{eq:trace.1}
  \lim_{t\to +\infty} \, \abs{f(t)-f(+\infty)}t^{a-\frac12 } = 0.
\end{equation}
\item\label{item:trace.gamma=1/2}
If $a=\frac{1}{2}$ for any $R>0$
\begin{equation}
  \label{eq:trace.R}
  \lim_{t\to R} \, \frac{\abs{f(t)-f(R)}}{\log\left(\frac{R}{t}\right)}= 0.
\end{equation}
\end{enumerate}
\end{proposition}

\begin{remark}
The function $r\in(0,+\infty)\mapsto r^a$ is $C^\infty(0,+\infty)$, hence the distribution $f'r^a$ is well defined. 
Equation \eqref{eq:ipo.der} has to be understood in the sense of distributions, i.e.~we will assume that there exists $C>0$ such that
for any test function $\varphi\in C^{\infty}_c(0,+\infty)$
\begin{equation}\label{eq:ipo.der.L^1_loc}
\abs{\langle f' r^a,\varphi\rangle}
\leq C ||\varphi||_{L^2}.
\end{equation}
Thanks to \eqref{eq:ipo.der.L^1_loc} and the density of $C^\infty_c$ in $L^2$ we get that there exists a unique linear and bounded functional $T:L^2\to \C$ that extends the linear functional $f' r^a$. 
By Riesz theorem, there exists a unique $g\in L^2$ such that $T=\langle\cdot,g\rangle_{L^2}$. 
In particular, for any test function $\varphi$  we get that
$\langle f' r^a,\varphi\rangle=\int g\overline{\varphi}$, that is $f'r^a=g$, which gives $f'=g r^{-a}\in L^1_{loc}(0,\infty)$ and \eqref{eq:ipo.der}.
\end{remark}

\begin{proof}
Let $0<\epsilon<M<+\infty$. From \eqref{eq:ipo.der} we get that $f'$ is integrable on $(\epsilon,M)$. 
Then \eqref{eq:fund.thm.calc} holds and so $f\in AC[\epsilon,M]$.  

\emph{\ref{item:hardy.gamma<1/2}} Let us assume $a<\frac{1}{2}$.   
By the H\"older inequality, we get that
\begin{equation}\label{eq:f'.L1(0,1)}
\int_0^1|f'(r)|\,dr\leq
 \left(\int_0^1 r^{-2a}\,dr \right)^{1/2} 
 \left( \int_0^ \infty|f'(r)|^2r^{2a}\,dr\right)^{1/2}<\infty,
\end{equation}
that is $f'\in L^1(0,1)$. Then $f\in AC[0,1]$ and \eqref{eq:fund.thm.calc} holds for $t,s\in[0,1]$. 
In particular, combining \eqref{eq:fund.thm.calc} and \eqref{eq:f'.L1(0,1)} we get that for $t\in (0,1]$:
\begin{equation*}
  \abs{f(t)-f(0)}
  \leq
   C
  t^{\frac12 - a} \left( \int_0^t |f'(r)|^2r^{2a}\,dr\right)^{1/2}.
\end{equation*}
Thanks to \eqref{eq:ipo.der} and by the absolute continuity of Lebesgue integral, \eqref{eq:trace.0} is proved.

\emph{\ref{item:hardy.gamma>1/2}} We assume now that $a>\frac{1}{2}$.
By the H\"older inequality, we get that
\begin{equation}\label{eq:f'.L1}
  \int_1^{+\infty}|f'(r)|\,dr\leq
  \left(\int_1^{+\infty} r^{-2a}\,dr \right)^{1/2} 
  \left( \int_0^{+\infty}|f'(r)|^2r^{2a}\,dr\right)^{1/2}<\infty,
\end{equation}
that is $f'\in L^1(1,+\infty)$. 
We will assume that $f$ is real-valued: for a complex-valued $f$  the same reasoning can be repeated for its real part and its imaginary part.
Let us fix $s\in [1,+\infty)$.
Since $a>\frac{1}{2}$, thanks to \eqref{eq:fund.thm.calc} and reasoning as in \eqref{eq:f'.L1} for any $t\in(1,+\infty)$ we get
\begin{equation}\label{eq:f(infty)}
|f(t)-f(s)|\leq \frac{s^{1/2-a}}{\sqrt{2a-1}}\left( \int_0^{+\infty}|f'(r)|^2r^{2a}\,dr\right)^{1/2} <+\infty.
\end{equation}
Thanks to the triangular inequality we can conclude that $f$ is bounded on $[1,+\infty)$. We set 
\begin{equation*}
f_-(+\infty):=\liminf_{r\to +\infty}f(r)>-\infty,\qquad
f_+(+\infty):=\limsup_{r\to +\infty}f(r)<+\infty.
\end{equation*}
Thanks to \eqref{eq:f(infty)} we get that
\begin{equation*}
f_+(+\infty)-f_-(+\infty)\leq|f_+(+\infty)-f(s)|+ |f_-(+\infty)-f(t)|\leq C s^{1/2-a}.
\end{equation*}
Since $a>\frac{1}{2}$, if $s\to +\infty$ in the previous expression, we get that $f_+(+\infty)=f_-(+\infty)=:f(+\infty)$. Finally  \eqref{eq:f(infty)} yields \eqref{eq:trace.1} too.

\emph{\ref{item:hardy.gamma=1/2}} In the last case $a=\frac{1}{2}$, equation \eqref{eq:trace.R} 
is proved  with the same approach used to prove \eqref{eq:trace.0}. 
\end{proof}

In the following Proposition we gather some weighted Hardy-type inequalities. 
Such results are very well known, but since we are focusing on the values of the function on the
boundaries of the integration domain, we give the proof for the sake of clarity.
We refer to \cite{kufner2017weighted} 
and \cite{machihara2016remarks} for details and references. 
\begin{proposition}\label{prop:hardy1dhardy}
Let $f$ be a distribution on $(0,+\infty)$ 
as in \Cref{prop:hardy1d}. 
Then the following hold: 
\begin{enumerate}[label=(\roman*)]
\item\label{item:hardy.gamma<1/2}
if $a<\frac{1}{2}$, then 
\begin{equation}\label{eq:hardy.gamma<1/2}
\left(a-\frac{1}{2}\right)^2
\int_0^{+\infty}\frac{|f(r)-f(0)|^2}{r^{2-2a}}\,dr
\leq 
\int_0^{+\infty}|f'(r)|^2 r^{2a}\,dr;
\end{equation}
\item\label{item:hardy.gamma>1/2}
if $a>\frac{1}{2}$ then
\begin{equation}\label{eq:hardy.gamma>1/2}
\left(a-\frac{1}{2}\right)^2
\int_0^{+\infty}\frac{|f(r)-f(+\infty)|^2}{r^{2-2a}}\,dr
\leq 
\int_0^{+\infty}|f'(r)|^2 r^{2a}\,dr;
\end{equation}
\item\label{item:hardy.gamma=1/2}
if $a=\frac{1}{2}$ then for any $R>0$ 
\begin{equation}\label{eq:hardy.gamma=1/2}
\frac{1}{4}
\int_0^{+\infty}\frac{|f(r)-f(R)|^2}{r\log^2\left(\frac{R}{r}\right)}\,dr
\leq 
\int_0^{+\infty}|f'(r)|^2 r\,dr.
\end{equation}
\end{enumerate}
\end{proposition}
\begin{remark}
The inequalities \eqref{eq:hardy.gamma<1/2}, \eqref{eq:hardy.gamma>1/2} and \eqref{eq:hardy.gamma=1/2} 
are sharp (in the sense that the constants on the left hand side cannot be improved) but they do not admit non-trivial extremizers. 
In fact, for $a\neq 1/2$ we set $f_a(r):=r^{\frac{1}{2}-a}$. Then 
\begin{equation}\label{eq:attained.hardy}
\lim_{\epsilon\to 0} \int_{\epsilon<|x|<1/\epsilon}\left(|f_a'(r)|r^{2a}-\frac{|f_a(r)|^2}{r^{2-2a}}\right) dr=0.
\end{equation}
Nevertheless $f_a$ does not verify \eqref{eq:ipo.der}, because $|f_a'(r)|^2r^a=\frac{1}{r}$ 
that is integrable neither close to $0$ nor to $+\infty$.
 This is the reason why we used the limiting formulation in \eqref{eq:attained.hardy}.
If $a=1/2$ the same argument can be repeated for $f_{1/2}(r):=\left(\log\left(\frac{R}{r}\right)\right)^{-1/2}$.
\end{remark}

\begin{proof}
\emph{\ref{item:hardy.gamma<1/2}} Let us assume $a<\frac{1}{2}$.   
Let $0< \epsilon< M$. With an explicit computation:
\begin{equation}\label{eq:square.gamma<1/2}
\begin{split}
0\leq&
\int_\epsilon^M
\abs*{f'(r)\,r^a+
\left(a-\frac{1}{2}\right)
\frac{f(r)-f(0)}{r^{1-a}}}^2\,dr\\
=& 
\int_\epsilon^M|f'|^2r^{2a}\,dr+
\left(a-\frac{1}{2}\right)^2 
\int_\epsilon^M \frac{|f(r)-f(0)|^2}{r^{2-2a}}\,dr\\
&+
\left(a-\frac{1}{2}\right)
2\Re\int_\epsilon^M
\frac{ f'(r)
\overline{(f(r)-f(0))}}{r^{1-2a}}\,dr.
\end{split}
\end{equation}
We integrate by parts the last term at right hand side:
since $a<\frac{1}{2}$, we can estimate from above neglecting the value on the boundary $M$, and we get that
\begin{equation}\label{eq:parts.gamma<1/2}
  \begin{split}
    \left(a-\frac{1}{2}\right) 
    2\Re\int_\epsilon^M \frac{ f'(r)
      \overline{(f(r)-f(0))}}{r^{1-2a}}\,dr
    = 
    \left(a-\frac{1}{2}\right)
    \int_\epsilon^M \frac{ \left(|f(r)-f(0)|^2\right)'}{r^{1-2a}}\,dr
    \\
    \leq
    -2\left(a-\frac{1}{2}\right)^2 \int_\epsilon^M
    \frac{|f(r)-f(0)|^2}{r^{2-2a}}\,dr
    -\left(a-\frac{1}{2}\right)
        \frac{|f(\epsilon)-f(0)|^2}{\epsilon^{1-2a}}.
  \end{split}
\end{equation}
Thanks to \eqref{eq:square.gamma<1/2}  and \eqref{eq:parts.gamma<1/2},  we get 
\begin{equation*}
\left(a-\frac{1}{2}\right)^2
\int_\epsilon^M\frac{|f(r)-f(0)|^2}{r^{2-2a}}\,dr+
\left(a-\frac{1}{2}\right)
    \frac{|f(\epsilon)-f(0)|^2}{\epsilon^{1-2a}}
\leq 
\int_\epsilon^M|f'|^2 r^{2a}\,dr.
\end{equation*}
Passing to the limit for $M\to+\infty$ and $\epsilon\to 0$, thanks to \eqref{eq:trace.0}, 
\eqref{eq:hardy.gamma<1/2} is proved.

\emph{\ref{item:hardy.gamma>1/2}} We assume now that $a>\frac{1}{2}$.
Let $0< \epsilon< M$. With an explicit computation:
\begin{equation}\label{eq:square.gamma>1/2}
\begin{split}
0\leq&
\int_\epsilon^M
\abs*{f'(r)\,r^a+
\left(a-\frac{1}{2}\right)
\frac{f(r)-f(+\infty)}{r^{1-a}}}^2\,dr\\
=&
\int_\epsilon^M|f'|^2r^{2a}\,dr+
\left(a-\frac{1}{2}\right)^2 
\int_\epsilon^M \frac{|f(r)-f(+\infty)|^2}{r^{2-2a}}\,dr\\
&+
\left(a-\frac{1}{2}\right)
2\Re\int_\epsilon^M
\frac{ f'(r)
\overline{f(r)-f(+\infty)}}{r^{1-2a}}\,dr.
\end{split}
\end{equation}
We integrate by parts the last term at right hand side:
since $a>\frac{1}{2}$, we can estimate from above neglecting the value on the boundary $\epsilon$, and we get that
\begin{equation}\label{eq:parts.gamma>1/2}
  \begin{split}
    \left(a-\frac{1}{2}\right) 
    2\Re\int_\epsilon^M \frac{ f'(r)
      \overline{(f(r)-f(+\infty))}}{r^{1-2a}}\,dr
    = 
    \left(a-\frac{1}{2}\right)
    \int_\epsilon^M \frac{ \left(|f(r)-f(+\infty)|^2\right)'}{r^{1-2a}}\,dr
    \\
    \leq
    -2\left(a-\frac{1}{2}\right)^2 \int_\epsilon^M
    \frac{|f(r)-f(+\infty)|^2}{r^{2-2a}}\,dr
    +\left(a-\frac{1}{2}\right)
\frac{|f(M)-f(+\infty)|^2}{M^{1-2a}}.
  \end{split}
\end{equation}
Thanks to \eqref{eq:square.gamma>1/2}  and \eqref{eq:parts.gamma>1/2},  we get 
\begin{equation*}
\left(a-\frac{1}{2}\right)^2
\int_\epsilon^M\frac{|f(r)-f(+\infty)|^2}{r^{2-2a}}\,dr-
\left(a-\frac{1}{2}\right)
\frac{|f(M)-f(+\infty)|^2}{M^{1-2a}}
\leq 
\int_\epsilon^M|f'|^2 r^{2a}\,dr.
\end{equation*}
Passing to the limit for $\epsilon\to 0$ and $M\to\infty$, thanks to \eqref{eq:trace.1} we get that
\eqref{eq:hardy.gamma>1/2} is proved.

\emph{\ref{item:hardy.gamma=1/2}} Let us finally consider the case $a=\frac{1}{2}$. Let $R>0$ and take $0<\epsilon <1<M$, such that $R\in [\epsilon,M]$. 
With explicit computations:
\begin{align*}
0\leq&
\int_\epsilon^M
\abs*{f'(r)\sqrt{r}-
\frac{1}{2} \frac{f(r)-f(R)}{\sqrt{r}\,\log\left(\frac{R}{r}\right)}}^2\,dr\\
=&
\int_\epsilon^M |f'(r)|^2r\,dr+
\frac{1}{4}\int_\epsilon^M \frac{|f(r)-f(R)|^2}{r\log^2 \left(\frac{R}{r}\right)} \,dr
-\frac{1}{2}\int_\epsilon^M \frac{(|f(r)-f(R)|^2)'}{\log \left(\frac{R}{r}\right)}\,dr.
\end{align*}
We integrate by parts and notice that 
the boundary contributions are negative, since $M > 1$ and $\epsilon < 1$. 
Consequently we get
\begin{equation*}
\frac{1}{4}\int_\epsilon^M \frac{|f(r)-f(R)|^2}{r\log^2 \left(\frac{R}{r}\right)}\,dr\leq \int_\epsilon^M |f'(r)|^2r\,dr.
\end{equation*}
Passing to the limit for $\epsilon\to 0$ and $M\to\infty$,  \eqref{eq:hardy.gamma=1/2} is proved.
\end{proof}
\section{Proof of Theorems \ref{thm:Hk.good}, \ref{thm:Hk.subandcritical}, \ref{thm:Hk.supercritical}}
We fix $j\in \seq{1/2,3/2,\dots}$, $m_j\in\seq{-j,\dots,j}$ and $k_j\in\seq{j+1/2,-j-1/2}$.
In this section we will simplify the notations and denote
\begin{equation}\label{eq:no.m_jk_j}
k:=k_j,\quad
\Phi^\pm:=\Phi^\pm_{m_j,k_j},\quad
f^\pm:=f^\pm_{m_j,k_j},\quad	
h^0:=h^0_{m_j, k_j},\quad
h:=h_{m_j, k_j},\quad
h^*:=h^*_{m_j, k_j}.
\end{equation}
We remind that $\mathring h$ is symmetric and its adjoint on $L^2(0,+\infty)^2$ is $h^*$.
In the following Proposition we give some details on the domain 
$\mathcal{D}(h^*)$.
\begin{proposition}\label{thm:caract:d(h*)}
Set 
$\delta:=(\lambda+k)^2+\mu^2-\nu^2$ and $\gamma:=\sqrt{|\delta|}$.
Then the following hold:
\begin{enumerate}[label=(\roman*)]
\item \label{item:properties.domain.gamma>1/2}
If $\delta > \frac{1}{4}$, then $ \mathcal{D}(h^*)=\mathcal{D}(h^0)$.
\item \label{item:properties.domain.gamma=1/2}
If $\delta=\frac{1}{4}$, then 
for all $(f^+,f^-)
\in \mathcal{D}(h^*)
$
we have
\begin{equation}\label{eq:det.gamma.geq1/2}
\liminf\limits_{r\to 0} 
f^+(r)  \overline{{f^-}(r)} 
=0.
\end{equation}
\item \label{item:properties.domain.gamma<1/2} 
If $0<\delta<\frac{1}{4}$, 
let
 $D \in \R^{2\times2}$ be the invertible matrix
\begin{equation*}
  D:= 
  \begin{cases}	
 \frac{1}{2\gamma(\lambda + k - \gamma)}
  \begin{pmatrix}
    \lambda + k - \gamma & \nu-\mu \\
     -(\nu + \mu)        & -(\lambda + k - \gamma)
  \end{pmatrix}
  \quad &\text{ if }\lambda + k - \gamma \neq 0,\\
  
    \frac{1}{-4\gamma^2}
    \begin{pmatrix}
      \mu - \nu            & 2\gamma \\
     2 \gamma & -(\nu + \mu) 
    \end{pmatrix}
    \quad &\text{ if }\lambda + k - \gamma = 0.
\end{cases}
\end{equation*}
Then for all
$(f^+,f^-)\in \mathcal{D}(h^*)$
 there exists
   $(A^+,A^-) \in \C^2$ 
such that
\begin{equation}
\label{eq:f0.gamma<1/2}
\begin{split}
  &\lim_{r\to 0}
\abs*{
    \begin{pmatrix}
    f^+(r) \\ f^-(r)
  \end{pmatrix}
  - D 
  \begin{pmatrix}
    A^+ r^\gamma \\ A^- r^{-\gamma}
  \end{pmatrix}
  }
r^{-1/2}
=0, \\
&  \int_0^{+\infty} 
\frac{1}{r^2}\abs*{
  \begin{pmatrix}
    f^+(r) \\ f^-(r)
  \end{pmatrix}
  - D 
  \begin{pmatrix}
    A^+ r^\gamma \\ A^- r^{-\gamma}
  \end{pmatrix}
}^2
\,dr
  < +\infty.
\end{split}
\end{equation}
Moreover, for any
$(f^+,f^-)\in \mathcal{D}(h^*)$
we have
\begin{equation}\label{eq:det.gamma.<1/2}
\lim_{r\to 0} 
\begin{vmatrix}
f^+(r) & \overline{\widetilde{f}^+(r)}\\
f^-(r) & \overline{\widetilde{f}^-(r)}
\end{vmatrix}
= 
\det(D)\cdot
\begin{vmatrix}
  A^+ & \overline{\widetilde{A}^+} \\
  A^- & \overline{\widetilde{A}^-}  
\end{vmatrix}
.
\end{equation}
\item \label{item:properties.domain.gamma=0}
If $\delta=0$, then 
let $M\in \R^{2\times2}$, $M^2=0$ defined as follows:
\begin{equation*}
   M:= 
\begin{pmatrix}
  -(k+\lambda) & - \nu + \mu \\
  \nu + \mu           & k+\lambda
\end{pmatrix}.
 \end{equation*}
Then for all 
$(f^+,f^-)\in \mathcal{D}(h^*)$
there exists $(A^+,A^-) \in \C^2$, 
such that
\begin{equation}
\label{eq:limit.f.gamma=0}
\begin{split}
&\lim_{r\to 0}
\abs*{
    \begin{pmatrix}
    f^+(r) \\ f^-(r)
  \end{pmatrix}
  - (M\log r+\mathbb{I}_2)
  \begin{pmatrix}
    A^+  \\ A^-
  \end{pmatrix}
 } 
 r^{-1/2}
 = 0, \\
  &\int_0^{+\infty} 
\frac{1}{r^2}\abs*{
  \begin{pmatrix}
    f^+(r) \\ f^-(r)
  \end{pmatrix}
  - (M\log r+\mathbb{I}_2)
  \begin{pmatrix}
    A^+  \\ A^- 
  \end{pmatrix}
}^2
\,dr
  < +\infty.
\end{split}
\end{equation}
Moreover, for any
$(\widetilde{f}^+,\widetilde{f}^-)\in \mathcal{D}(h^*)$
we have
we have
\begin{equation}\label{eq:det.gamma.=0}
  \lim_{r\to 0}
    \begin{vmatrix}
     f^+ & \overline{\widetilde{f^+}} \\
     f^- & \overline{\widetilde{f^-}}
  \end{vmatrix}
  =
  \begin{vmatrix}
    A^+ & \overline{ \widetilde A^+} \\ 
    A^- & \overline{\widetilde {A}^-} 
  \end{vmatrix}.
\end{equation}
\item If $\delta<0$ let
 $D \in \C^{2\times2}$ be the invertible matrix
\begin{equation*}
  D:= 
   \frac{1}{2i \gamma(\lambda + k - i\gamma)}
  \begin{pmatrix}
    \lambda + k - i\gamma & \nu-\mu \\
     -(\nu + \mu)        & -(\lambda + k - i\gamma)
  \end{pmatrix}.
\end{equation*}
Then for all
$(f^+,f^-)\in \mathcal{D}(h^*)$
 there exists
   $( A^+,A^-)\in \C^2$ 
such that
\begin{equation}
\label{eq:f0.gamma<0}
 \begin{split}
 & \lim_{r\to 0}
\abs*{
    \begin{pmatrix}
    f^+(r) \\ f^-(r)
  \end{pmatrix}
  - D 
  \begin{pmatrix}
    A^+ r^{i\gamma} \\ A^- r^{-i\gamma}
  \end{pmatrix}
  }
r^{-1/2}
=0, \\
&
  \int_0^{+\infty} 
\frac{1}{r^2}\abs*{
  \begin{pmatrix}
    f^+(r) \\ f^-(r)
  \end{pmatrix}
  - D 
  \begin{pmatrix}
    A^+ r^{i\gamma} \\ A^- r^{-i\gamma}
  \end{pmatrix}
}^2
\,dr
  < +\infty.
  \end{split}
\end{equation}
Moreover, for any
$(\widetilde{f}^+,\widetilde{f}^-)\in \mathcal{D}(h^*)$
we get
\begin{equation}\label{eq:det.gamma.<0}
\lim_{r\to 0} 
\begin{vmatrix}
f(r) & \overline{\widetilde{f}(r)}\\
g(r) & \overline{\widetilde{g}(r)}
\end{vmatrix}
= 
 \frac{1}{2i\gamma(\mu^2-\nu^2)}\cdot
\begin{vmatrix}
  A^+ & (\nu-\mu)\overline{\widetilde{A^-}} \\
  A^- & (\nu+\mu)\overline{\widetilde{A^+}}  
\end{vmatrix}.
\end{equation}
\end{enumerate}
\end{proposition}
\begin{proof}
We start noticing that for a general $ (f^+, f^-)
    \in \mathcal{D}(h^*),$
using the matrix representation of $h^*$ defined in \eqref{eq:dirac.spherical*}, we can deduce that
\begin{equation}\label{eq:f.inL^2:hf.inL^2}
\begin{pmatrix}
\partial_r+ \frac{k+\lambda}{r}&\frac{\nu-\mu}{r}\\
-\frac{\nu+\mu}{r}&\partial_r- \frac{k+\lambda}{r}\\
\end{pmatrix}
  \begin{pmatrix}
    f^+(r) \\ f^-(r)
  \end{pmatrix}
\in L^2(0,+\infty)^2.
\end{equation}
Set
\begin{equation}
\sqrt{\delta}:=
\begin{cases}
\gamma&\text{if}\ \delta\geq 0,\\
i\gamma&\text{if}\ \delta<0.
\end{cases}
\end{equation}
We consider the matrices
\begin{equation}\label{eq:rappr.M}
\begin{pmatrix}
  -(k+\lambda-\sqrt{\delta}) & - \nu + \mu \\
  \nu + \mu           & k+\lambda-\sqrt{\delta}
\end{pmatrix},
\quad
\begin{pmatrix}
  -\nu-\mu & -(k+\lambda+\sqrt{\delta}) \\
  -(k+\lambda+\sqrt{\delta}) & -\nu+\mu
\end{pmatrix}.
\end{equation}
In the case $\delta>0$ at
least one matrix in \eqref{eq:rappr.M}
is invertible: let $M$ be the first matrix if this is invertible
and the second otherwise.
In the case $\delta=0$
we can choose $M$ to be the first or the second one (in fact they are unitarily equivalent): we choose the first one.
Finally, in the case $\delta<0$ we can choose $M$ to be the first or the second one (in fact they are both invertible and unitarily equivalent): we choose the first one. Setting
\begin{equation}
 \label{eq:phi=Mf}
\begin{pmatrix}   
    \varphi^+(r) \\\varphi^-(r)
  \end{pmatrix}
  :=
   M 
  \begin{pmatrix}
    f^+(r) \\ f^-(r)
  \end{pmatrix}
\end{equation}
we get with an easy computation 
  \begin{equation}\label{eq:dopocommuta}
\begin{pmatrix}
\partial_r-\frac{\sqrt{\delta}}{r}&0\\
0&\partial_r+\frac{\sqrt{\delta}}{r}
\end{pmatrix}\cdot
\begin{pmatrix}   
    \varphi^+ \\\varphi^-
\end{pmatrix}
=
M\cdot
\begin{pmatrix}
\partial_r+ \frac{k+\lambda}{r}&\frac{\nu-\mu}{r}\\
-\frac{\nu+\mu}{r}&\partial_r- \frac{k+\lambda}{r}\\
\end{pmatrix}
  \begin{pmatrix}
    f^+ \\ f^-
  \end{pmatrix}.
  \end{equation}
Moreover it is easy to observe that, for all $a \in \C$ and $f$ regular enough we have 
\begin{equation}\label{eq:derivata.f/r}
  \left( \partial_r + \frac{a}{r} \right) f(r)=\left(\partial_r(r^a f)\right)r^{-a}.
\end{equation}
Combining \eqref{eq:f.inL^2:hf.inL^2}, \eqref{eq:dopocommuta} and \eqref{eq:derivata.f/r}
 we have
\begin{equation}\label{eq:radialigen}
     \int_0^{+\infty} \abs{r^{\sqrt{\delta}}\partial_r(r^{-\sqrt{\delta}}\varphi^+(r))}^2    \,dr 
     +
     \int_0^{+\infty}   \abs{r^{-\sqrt{\delta}}\partial_r(r^{\sqrt{\delta}}\varphi^-(r))}^2  \,dr
    <+\infty.
\end{equation}
We assume now $\delta\geq 0$, that is $\sqrt{\delta}=\gamma$. 
In this case $M$ is a real matrix.

From \eqref{eq:radialigen} we deduce that
\begin{equation}\label{eq:radiali}
     \int_0^{+\infty} r^{2\gamma} \abs{\partial_r(r^{-\gamma}\varphi^+(r))}^2    \,dr 
     +
     \int_0^{+\infty} r^{-2\gamma}  \abs{\partial_r(r^{\gamma}\varphi^-(r))}^2  \,dr
     <+\infty.
 \end{equation}
We can immediately get informations on the function $\varphi^-$. Indeed, 
 $r^{\gamma}\varphi^-$ is in 
$L^1_{loc}(0,+\infty)\cap L^1(0,1)$: 
choosing $a=-\gamma\leq 0$ in 
\emph{\ref{item:hardy.gamma<1/2}} of \Cref{prop:hardy1d} we get that 
$\varphi^- \in C[0,+\infty)$ and there exists a constant $A^- \in \C$, depending on $\varphi^-$, such that
\begin{equation}\label{eq:lim-}
  \lim_{r\to 0} \, \abs{\varphi^-(r)-A^-r^{-\gamma}}r^{-\frac12} = 0.
\end{equation}
Moreover,
thanks to \eqref{eq:hardy.gamma<1/2},
we get
\begin{equation}\label{eq:cond.C}
  \int_0^{+\infty} \frac{\abs{\varphi^-(r) - A^- r^{-\gamma}}^2}{r^2}\,dr 
  \leq  \frac{4}{(2\gamma +1)^2} \int_0^{+\infty}  r^{-2\gamma} \abs{\partial_r(r^{\gamma}\varphi^-(r))}^2  \,dr
  < +\infty.
\end{equation}

In order to get informations on the function $\varphi^+$, we need to distinguish various cases,
depending on the size of $\gamma$.
\subsubsection*{Case $\gamma>1/2$}
Since $\gamma > 1/2$, we have that $r^{-\gamma}\varphi^+$ is in 
$L^1_{loc}(0,+\infty)\cap L^1(1,+\infty)$: 
 choosing $a=\gamma$ in
\emph{\ref{item:hardy.gamma>1/2}} of \Cref{prop:hardy1d},
we get
\begin{equation*}
  \lim_{r\to +\infty} \, \abs{\varphi^+(r)}r^{-\frac12} = 0,
\end{equation*}
observing that under our assumptions $\varphi^+ (+\infty)=0$.
Thanks to \eqref{eq:hardy.gamma>1/2}
and from 
\eqref{eq:radiali} we have that
\begin{equation}\label{eq:gamma>12+}
    \int_0^{+\infty} \frac{\abs{\varphi^+(r)}^2}{r^{2}}\,dr 
    = \int_0^{+\infty} \frac{\abs{r^{-\gamma}\varphi^+(r)}^2}{r^{2-2\gamma}}\,dr
    \leq \frac{4}{(2\gamma-1)^2} 
     \int_0^{+\infty} r^{2\gamma} \abs{\partial_r(r^{-\gamma}\varphi^+(r))}^2    \,dr
     <+ \infty,
   \end{equation}

Moreover, since $\varphi^- \in L^2(0,+\infty)$ behaves like $A^-r^{-\gamma}$ 
next to the origin (i.e.~\eqref{eq:cond.C} holds), we have that 
\begin{equation}
  \label{eq:Cmustbe0}
  \begin{split}
    \int_0^1 \frac{\abs{A^-}^2}{r^{2\gamma}}\,dr 
    &\leq 
    2 \int_0^1 \abs{\varphi^-(r)}^2 \,dr 
    + 2 \int_0^1 \abs{\varphi^-(r) - A^- r^{-\gamma}}^2 \,dr 
    \\
    \leq & 
    2 \int_0^{+\infty} \abs{\varphi^-(r)}^2 \,dr 
    + 2 \int_0^{+\infty} \frac{\abs{\varphi^-(r) - A^- r^{-\gamma}}^2 }{r^2}\,dr
    < + \infty.
  \end{split}
\end{equation}
Since $\gamma>1/2$, necessarily this implies $A^-=0$ in \eqref{eq:cond.C}.
Combining \eqref{eq:radiali}, \eqref{eq:cond.C} (for $A^-=0$) and \eqref{eq:gamma>12+}  we can conclude,
thanks to the invertibility of $M$,
\begin{equation}\label{eq:fginD(r^-1/2)}
\int_0^{+\infty} \frac{|f^+(r)|^2}{r^2}\,dr
+
\int_0^{+\infty} \frac{|f^-(r)|^2}{r^2}\,dr
  < + \infty.
\end{equation}
Thanks to \eqref{eq:liberocontmax}, we get $\D(h^0)\subset \D(h^*)$.
From \eqref{eq:fginD(r^-1/2)} and the by the definition of $\D(h^*)$ (see \eqref{eq:dirac.spherical*}) we get that $\left(\de_r\pm\frac{k}{r}\right)f^\pm\in L^2(0,+\infty)$ and so $\D(h^*)=\D(h^0)$.
\subsubsection*{Case $\gamma=1/2$}
Reasoning as in the previous step, we get that 
\eqref{eq:cond.C} holds for $A^-=0$.
Thanks to
\emph{\ref{item:hardy.gamma=1/2}} of  \Cref{prop:hardy1d}
we have that $\varphi^+ \in C(0,+\infty)$ and by \eqref{eq:hardy.gamma=1/2}
\begin{equation*}
  \int_0^{1/2}\frac{ |\varphi^+(r)|^2}{r^2\log^2 \left(\frac1r\right)}\,dr 
  =
  \int_0^{1/2}\frac{|r^{-1/2}\varphi^+(r)|^2}{r^2\log^2 \left(\frac1r\right)}\,dr
  \leq 4
  \int_0^{+\infty} r\abs{\partial_r(r^{-1/2}\varphi^+(r))}^2    \,dr 
  + R
  < +\infty,
\end{equation*}
for $R>0$ a finite constant,
that implies that
\begin{equation}\label{eq:liminfgamma=12}
    \liminf\limits_{r\to 0} \frac{\abs{\varphi^+(r)}}{r^{1/2}\log(1/r)} =0.
\end{equation}
We can conclude \eqref{eq:det.gamma.geq1/2} thanks to 
\eqref{eq:lim-} (with $A^-=0$) and \eqref{eq:liminfgamma=12}, 
remarking the property of the inferior limit:
\begin{equation*}
  \liminf_{x\to x_0} (f(x)g(x)) = \left( \liminf_{x\to x_0}f(x)\right) \left(\lim_{x\to x_0} g(x)\right),
\end{equation*}
when $\lim_{x\to x_0} g(x)$ exists.
\subsubsection*{Case $0<\gamma< 1/2$}
On these terms $r^{-\gamma}\varphi^+$ is in 
$L^1_{loc}(0,+\infty)\cap L^1(0,1)$. 
Choosing $a=\gamma$ in \emph{\ref{item:hardy.gamma<1/2}} of \Cref{prop:hardy1d}
we have that $\varphi^+ \in C[0,+\infty)$ and 
there exists a constant $A^+ \in \C$, depending on $\varphi^+$, such that
\begin{equation}\label{eq:lim+}
  \lim_{r\to 0} \, \abs*{\varphi^+(r) - A^+ r^{\gamma} } r^{-\frac12} = 0,
\end{equation}
and moreover, by \eqref{eq:hardy.gamma<1/2},
we get
\begin{equation}\label{eq:cond.B}
  \int_0^{+\infty} \frac{\abs{\varphi^+(r) - A^+ r^{\gamma}}^2}{r^2}\,dr
  \leq \frac{4}{(2\gamma-1)^2}
     \int_0^{+\infty} r^{2\gamma} \abs{\partial_r(r^{-\gamma}\varphi^+(r))}^2    \,dr
     < +\infty.
\end{equation}
We set $D:=M^{-1}$. 
Thanks to \eqref{eq:phi=Mf}, \eqref{eq:lim-}, \eqref{eq:lim+} we get the first equation in 
\eqref{eq:f0.gamma<1/2}. Moreover thanks to 
\eqref{eq:radiali}, \eqref{eq:cond.C} and \eqref{eq:cond.B} we get
the second equation in \eqref{eq:f0.gamma<1/2}.
Finally
\begin{equation}
  \label{eq:11}
  \begin{split}
    \det(M)\cdot
    \begin{vmatrix}
      f^+(r) & \overline{\widetilde{f^+}(r)} \\
      f^-(r) & \overline{\widetilde{f^-}(r)} \\
    \end{vmatrix}
    =&
    \begin{vmatrix}
      \varphi^+ (r) & \overline{\widetilde \varphi^+ (r)}  \\  
      \varphi^-(r)  &  \overline{\widetilde \varphi^-(r) }
    \end{vmatrix}
      =\varphi^+ (r)  \overline{\widetilde \varphi^-(r) }
  - \varphi^-(r)  \overline{\widetilde \varphi^+ (r)} 
  \\
  = &\varphi^+(r) \overline{(\widetilde \varphi^- (r) - \widetilde A^- r^{-\gamma})}
  - (\varphi^-(r) - A^- r^{-\gamma}) \overline{\widetilde \varphi^+(r) } 
  \\
  &
    + (\varphi^+(r)  - A^+ r^{\gamma})\overline{\widetilde A^-} r^{-\gamma}
  - A^- r^{-\gamma} \overline{( \widetilde \varphi^+(r)  - \widetilde A^+ r^{\gamma})}
  \\
  &
  + A^+\overline{\widetilde A^-} - A^- \overline{\widetilde A^+}.
\end{split}
\end{equation}
Thanks to \eqref{eq:phi=Mf}, \eqref{eq:lim-}, \eqref{eq:lim+}
observing that the first four terms at right hand side are infinitesimal for $r \to 0$, we can conclude \eqref{eq:det.gamma.<1/2}.
\subsubsection*{Case $\gamma=0$}
We recall that, in this case,  the two possibilities we give for the matrix $M$ in \eqref{eq:rappr.M} are unitarily equivalent. For this reason we will always choose the first one, that is
\begin{equation*}
M:=\begin{pmatrix}
-(k+\lambda)&-\nu+\mu\\
\nu+\mu&k+\lambda
\end{pmatrix}.
\end{equation*}
We remind that \eqref{eq:f.inL^2:hf.inL^2} now reads
\begin{equation}\label{eq:daunire1}
  \begin{pmatrix}    f^+ \\ f^-   \end{pmatrix}'
  - \frac{1}{r}
  \begin{pmatrix}    \varphi^+ \\ \varphi^-   \end{pmatrix}
  \in L^2(0,\infty)^2.
\end{equation}
Moreover, choosing $a=0$ in 
\emph{\ref{item:hardy.gamma<1/2}} of \Cref{prop:hardy1d} we get from 
\eqref{eq:radiali} 
that $(\varphi^+,\varphi^-) \in C[0,+\infty)^2$ and 
there exists $(B^+, B^-) \in \C^2$, such that
\begin{equation*}
  \lim_{r\to 0} \abs*{
    \begin{pmatrix} \varphi^+(r) \\ \varphi^-(r) \end{pmatrix}
    -
    \begin{pmatrix} B^+ \\ B^- \end{pmatrix}
  } r^{-1/2}=0.
\end{equation*}
Moreover by \eqref{eq:hardy.gamma<1/2},
we get
\begin{equation*}
  \begin{split}
    \int_0^{+\infty} \frac{1}{r^2}
      \abs*{\begin{pmatrix} \varphi^+(r) \\ \varphi^-(r) \end{pmatrix}
        -
        \begin{pmatrix} B^+ \\ B^- \end{pmatrix}}^2
    \,dr
    \leq & 4
    \int_0^{+\infty} 
    \abs*{
      \partial_r 
      \begin{pmatrix}
        \varphi^+(r) \\ \varphi^-(r)
      \end{pmatrix}
      }^2 
    < + \infty.
  \end{split}
\end{equation*}
In particular, this shows that
\begin{equation}
\label{eq:daunire2}
\frac{1}{r}\left(
{    \begin{pmatrix} \varphi^+ \\ \varphi^- \end{pmatrix} - \begin{pmatrix}B^+ \\ B^-\end{pmatrix}}\right)
    \in L^2(0,+\infty)^2.
\end{equation}
Thanks to \eqref{eq:daunire1} and \eqref{eq:daunire2} we get that 
\begin{equation*}
\left[
  \begin{pmatrix}
    f^+ \\f^-
  \end{pmatrix}
  -
  \begin{pmatrix}
    B^+ \\ B^-
  \end{pmatrix}
  \log r
\right]'
  \in L^2(0,+\infty)^2. 
\end{equation*}
Applying again  
\emph{\ref{item:hardy.gamma<1/2}} of \Cref{prop:hardy1d} with $a=0$ we get that 
$f^\pm-B^\pm\log r \in C[0,+\infty)$ and 
there exist constants $A^\pm\in \C$,  such that
\begin{equation}\label{eq:completa1}
\lim_{r\to 0}
\abs*{
  \begin{pmatrix}
    f^+(r) \\ f^-(r)
  \end{pmatrix}
  -
  \begin{pmatrix}
    B^+ \\ B^-
  \end{pmatrix}
  \log r
  -
  \begin{pmatrix}
    A^+ \\ A^-
  \end{pmatrix}
}
r^{-1/2}=0,
\end{equation}
moreover, by \eqref{eq:hardy.gamma<1/2},
we get
\begin{equation}\label{eq:completa}
\int_0^{+\infty} 
\frac{1}{r^2}
\abs*{
  \begin{pmatrix}
    f^+(r) \\ f^-(r)
  \end{pmatrix}
  -
  \begin{pmatrix}
    B^+ \\ B^-
  \end{pmatrix}
  \log r
  -
  \begin{pmatrix}
    A^+ \\ A^-
  \end{pmatrix}
}^2dr<+\infty.
\end{equation}
Since $M^ 2=0$, from \eqref{eq:phi=Mf} and \eqref{eq:daunire2} we get
\begin{equation*}
  \frac{1}{r} M \begin{pmatrix}B^+ \\ B^-\end{pmatrix} =
\frac{1}{r} M \left(
{    \begin{pmatrix} \varphi^+ \\ \varphi^- \end{pmatrix} - \begin{pmatrix}B^+ \\ B^-\end{pmatrix}}\right)
    \in L^2(0,+\infty)^2,
\end{equation*}
 that implies $M (B^+ \, B^-)^t =0 $.
As a consequence, from \eqref{eq:completa} we get that 
\begin{equation}\label{eq:cond.C.gamma=0}
  \frac{1}{r}\left[ 
   \begin{pmatrix}
     \varphi^+ \\ \varphi^-
   \end{pmatrix}
   - M 
   \begin{pmatrix}
     A^+ \\ A^-
   \end{pmatrix}
   \right] 
=
 \frac{1}{r}\left[M 
   \begin{pmatrix}
     f^+ \\ f^-
   \end{pmatrix}
   - M 
   \begin{pmatrix}
     A^+ \\ A^-
   \end{pmatrix}
   \right] 
    \in L^2(0,+\infty)^2.
\end{equation}
Such a condition and \eqref{eq:daunire2} gives that 
\begin{equation*}
  \begin{pmatrix}
    B^+ \\ B^- 
  \end{pmatrix}
  =
  M
  \begin{pmatrix}
    A^+ \\ A^-
  \end{pmatrix},
\end{equation*}
that lets us conclude \eqref{eq:limit.f.gamma=0} thanks to \eqref{eq:completa1}.

In order to exploit the linearity of the determinant in the columns,
in the following we commit abuse of notation, denoting
\begin{equation}\label{eq:abuse}
  \abs*{
    \begin{pmatrix}
      a \\ c
    \end{pmatrix}
    \begin{pmatrix}
      b \\ d
    \end{pmatrix}
  }
  :=
  \begin{vmatrix}
    a & b \\
    c & d
  \end{vmatrix}.
\end{equation}
We have that
\begin{equation*}
\begin{split}
  \abs*{
  \begin{matrix}
       f^+(r) &   \overline{\widetilde{f^+}(r)}
       \\ f^-(r) &\overline{\widetilde{f^-}(r)}
  \end{matrix}}
  =&
  \abs*{
    \begin{pmatrix}
      f^+(r) \\ f^-(r)
    \end{pmatrix}
    - (M \log r + \mathbb{I}_2) 
    \begin{pmatrix}
      A^+ \\ A^-
    \end{pmatrix}
    \quad
    \overline{\begin{pmatrix}
      \widetilde{f^+}(r) \\ \widetilde{f^-}(r)
    \end{pmatrix}
    - (M \log r + \mathbb{I}_2) 
    \begin{pmatrix}
      \widetilde A^+ \\ \widetilde A^-
    \end{pmatrix}}
    } \\
    &+ 
  \abs*{
     (M \log r + \mathbb{I}_2) 
    \begin{pmatrix}
      A^+ \\ A^-
    \end{pmatrix}
    \quad
    \overline{\begin{pmatrix}
      \widetilde{f^+}(r) \\ \widetilde{f^-}(r)
    \end{pmatrix}
    - (M \log r + \mathbb{I}_2) 
    \begin{pmatrix}
      \widetilde A^+ \\ \widetilde A^-
    \end{pmatrix}}
    } 
    \\
    &+
      \abs*{
    \begin{pmatrix}
      f^+(r) \\ f^-(r)
    \end{pmatrix}
    - (M \log r + \mathbb{I}_2) 
    \begin{pmatrix}
      A^+ \\ A^-
    \end{pmatrix}
    \quad
    \overline{
     (M \log r + \mathbb{I}_2) 
    \begin{pmatrix}
      \widetilde A^+ \\ \widetilde A^-
    \end{pmatrix}}
    } \\
    &+
  \abs*{
    (M \log r + \mathbb{I}_2) 
    \begin{pmatrix}
      A^+ \\ A^-
    \end{pmatrix}
    \quad
    \overline{
     (M \log r + \mathbb{I}_2) 
    \begin{pmatrix}
      \widetilde A^+ \\ \widetilde A^-
    \end{pmatrix}}
    }.
  \end{split}
\end{equation*}
Since $M^2=0$ we get $\det(\mathbb{I}_2+M\log r)=1$. Thanks to the first equation in \eqref{eq:limit.f.gamma=0}, the first three terms at 
right hand side tend to 0 as $r \to 0$, and  we can conclude \eqref{eq:det.gamma.=0}.

\subsubsection{Case $\delta < 0$}
We $\sqrt{\delta}=i\gamma$. 
In this case $M$ is an invertible complex matrix with inverse $D:=M^{-1}$ given by \eqref{eq:defn.D.negative}. 
Denoting with $\overline D$ the complex conjugate matrix of $D$ we have
\begin{equation}\label{eq:condD^2}
D^2=\frac{1}{-2i\gamma(k+\lambda-i\gamma)} \I_2,\quad
D\overline{D}\,=\frac{1}{2i\gamma(\nu^2-\mu^2)}
\begin{pmatrix}
0&\nu-\mu\\
\nu+\mu&0
\end{pmatrix}.
\end{equation}

Since $|r^{\pm i\gamma}|=1$, from \eqref{eq:radialigen} we deduce
\begin{equation*}
     \int_0^{+\infty} \abs{\partial_r(r^{-i\gamma}\varphi^+(r))}^2    \,dr 
     +
     \int_0^{+\infty} \abs{\partial_r(r^{i\gamma}\varphi^-(r))}^2  \,dr<+\infty.
 \end{equation*}
Choosing $a=0$ in 
\emph{\ref{item:hardy.gamma<1/2}} of \Cref{prop:hardy1d} we get that 
$r^{\mp i \gamma }\varphi^{\pm} \in C[0,+\infty)$ and there exist two constants $A^\pm \in \C$, depending on $\varphi^{\pm}$, such that
\begin{equation}\label{eq:limpm}
  \lim_{r\to 0} \, \abs{\varphi^\pm(r)-A^\pm r^{\pm i\gamma}}r^{-\frac12} = 0.
\end{equation}
Moreover
by \eqref{eq:hardy.gamma<1/2},
we get
\begin{equation}\label{eq:cond.Cpm} 
  \int_0^{+\infty} \frac{\abs{\varphi^\pm(r) - A^\pm r^{\pm i\gamma}}^2}{r^2}\,dr 
  \leq  {4} \int_0^{+\infty}  \abs{\partial_r(r^{\mp i\gamma}\varphi^\pm(r))}^2  \,dr
  < \infty.
\end{equation}
We deduce \eqref{eq:f0.gamma<0}
from \eqref{eq:phi=Mf}, \eqref{eq:limpm}, \eqref{eq:cond.Cpm}.
Finally, with the abuse of notations in \eqref{eq:abuse}, 
 from \eqref{eq:condD^2} we get
\begin{equation}\label{eq:112}
\begin{split}
 \abs*{\begin{matrix}
       f^+(r) &   \overline{\widetilde{f^+}(r)}
       \\ f^- (r)&\overline{\widetilde{f^-}(r)}
  \end{matrix}}
  =&
  \abs*{
  D\begin{pmatrix}
  \varphi^+(r)\\
  \varphi^-(r)
  \end{pmatrix}
  \overline{
   D\begin{pmatrix}
  \widetilde{\varphi^+}(r)\\
 \widetilde{ \varphi^-}(r)
  \end{pmatrix}
  }}=
  \frac{1}{\det D}
	\abs*{
	D^2\begin{pmatrix}
  \varphi^+(r)\\
  \varphi^-(r)
  \end{pmatrix}
  D\overline{D}
  \begin{pmatrix}
  \overline{\widetilde{\varphi^+}(r)}\\
 \overline{\widetilde{ \varphi^-}(r)}
  \end{pmatrix}
  }  \\
  =& 
  \frac{1}{2i\gamma(\mu^2-\nu^2)}
   \abs*{\begin{matrix}
       \varphi^+(r) &   (\nu-\mu)\overline{\widetilde{\varphi^-}(r)}
       \\ \varphi^-(r) &(\nu+\mu)\overline{\widetilde{\varphi^+}(r)}
  \end{matrix}}.
  \end{split}
  \end{equation}
We prove immediately \eqref{eq:det.gamma.<0} from \eqref{eq:112}, reasoning as 
in the proof of \eqref{eq:11}.
\end{proof}

We can now finally prove Theorems \ref{thm:Hk.good}, \ref{thm:Hk.subandcritical}, \ref{thm:Hk.supercritical}.
\begin{proof}[Proof of \Cref{thm:Hk.good}]
\emph{\ref{thm:Hk.gamma>12}} 
\,  Thanks to\emph{ \ref{item:properties.domain.gamma>1/2} }in  \Cref{thm:caract:d(h*)},
  we already know that 
  \begin{equation*}
 \mathcal{D}(h^*) =\mathcal{D}(h^0).
  \end{equation*}
  This gives immediately that $h^*$ is symmetric, that is $h$ is essentially self-adjoint on $C^\infty_c(0,+\infty)^2$.
  
\emph{\ref{thm:Hk.gamma=12}} 
We show that $h^*$ is symmetric on $\mathcal{D}(h^*)$, that implies the essential self-adjointness of $h$.
Indeed
for all $(f^+, f^-) \in \mathcal{D}(h^*)$ we have
\begin{equation}
\label{eq:is.it.symmetric}
\begin{split}
     \int_0^{+\infty} h^* (f^+,f^-) \cdot &\overline{(f^+,f^-)}\,dr
    -  \int_0^{+\infty}  (f^+,f^-) \cdot \overline{h^* (f^+,f^-)}\,dr
    \\
    &=  \lim_{n} \int_{\epsilon_n}^{+\infty} h^* (f^+,f^-) \cdot \overline{(f^+,f^-)}\,dr
    -  \int_{\epsilon_n}^{+\infty}  (f^+,f^-) \cdot \overline{h^* (f^+,f^-)}\,dr
    \\
    &=-\lim_{n} 
    \begin{vmatrix}
      f^+({\epsilon_n}) & \overline{f^+({\epsilon_n})}\\
      f^-({\epsilon_n}) & \overline{f^-({\epsilon_n})}\\
    \end{vmatrix},
    \end{split}
\end{equation}
for any $(\epsilon_n)_n$, $\epsilon_n \to 0$.
The limit in \eqref{eq:is.it.symmetric} exists 
for every choice of the sequence $(\epsilon_n)_n$, $\epsilon_n \to 0$, since $(f^+,f^-)\in \mathcal{D}(h^*)$.
Moreover, taking the sequence associated to the inferior limit,
 it vanishes thanks to \eqref{eq:det.gamma.geq1/2}.
Finally, it is easy to show that $\mathcal{D}(h^0)  \subset \mathcal{D}(h^*)$.
\end{proof}

For the proof of  \Cref{thm:Hk.subandcritical} we will need the following Lemma.
\begin{lemma}\label{lem:determinat.matrix.null}
Let $V$ be a complex proper subspace of $\C^2$. Then the following are equivalent:
\begin{enumerate}[label=(\roman*)]
\item\label{item:i.determinant}
$(A^+,A^-)\in V$ if and only if
$\begin{vmatrix}
A^+& \overline {A^+}\\
A^-& \overline {A^-}
\end{vmatrix}
=0$,
\item\label{item:ii.determinant}
$(A^+,A^-)\in V$ if and only if 
$
A^+ \overline{A^-}\in\R
$,
\item\label{item:iii.determinant}
$V=\seq{(0,0)}$ or $V=V_\theta:=\seq*{
(A^+,A-)\in\C^2:\,
A^+\sin\theta
+
A^-\cos\theta=0}
$,
for $\theta\in[0,\pi)$.
\end{enumerate} 
\end{lemma}
\begin{proof}
It is easy to prove that 
\emph{\ref{item:i.determinant}} is equivalent to \emph{\ref{item:ii.determinant}} and that
\emph{\ref{item:iii.determinant}} implies \emph{\ref{item:ii.determinant}}. Let us prove that \emph{\ref{item:ii.determinant}} implies \emph{\ref{item:iii.determinant}}. Let $V$ be as in \emph{\ref{item:ii.determinant}}:
$V$ can not be the whole $\C^2$, so $V$ is a proper subspace of $\C^2$, 
i.e. it has dimension zero or one. 
In the first case $V=\seq{(0,0)}$. Let us suppose now that $V$ has dimension one, that is $V=\langle(A^+_0,A^-_0)\rangle$ for some $(A^+_0,A^-_0)\neq (0,0)$
 with $A^+_0\overline{A^-_0}\in \R$. 
Using polar coordinates we get $A^+_0=u e^{i s}$ and $A^-_0=v e^{i t}$, 
then 
$A^+_0\overline{A^-_0}=uv e^{i (s-t)}$ which implies that $s=t$ or $s=t+\pi$, that is equivalent to say that there are $p,q\in\R$, 
$(p,q)\neq (0,0)$ such that $pA^+_0+qA^-_0=0$. 
We can always suppose that $p\geq 0$ (otherwise we replace $(p,q)$ with $(-p,-q)$) and $|p|^2+|q|^2=1$ 
(otherwise we replace $(p,q)$ with $(p^2+q^2)^{-1/2}(p,q)$). Then $p=\sin\theta$ and $q=\cos\theta$ for $\theta\in[0,\pi)$.
\end{proof}

\begin{proof}[Proof of \Cref{thm:Hk.subandcritical}]
\emph{\ref{thm:Hk.gamma<12}}
   Let $t$ be a self adjoint extension of $h$, that is $h\subseteq t = t^* \subseteq h^*$.
   Thanks to \emph{\ref{item:properties.domain.gamma<1/2}} in \Cref{thm:caract:d(h*)},
   we have that for all $(f^+,f^-) \in \mathcal{D}(t)$
   there exist constants $A^\pm \in \C$
   such that 
   \begin{equation*}
     \lim_{r\to 0} 
     \abs*{
     \begin{pmatrix} f^+(r) \\ f^-(r) \end{pmatrix} 
     - D 
     \begin{pmatrix}
       A^+ r^{\gamma} \\
       A^- r^{-\gamma}
     \end{pmatrix}}
            r^{-1/2}
     =0,
   \end{equation*}
   where $D$ is the invertible real matrix defined in \eqref{eq:defn.D}.
   Moreover, the map $(f^+,f^-) \in \mathcal{D}(t)\mapsto (A^+,A^-)\in \C^2$ 
   is a homomorphism of linear spaces, thus its image is a linear subspace of $\C^2$: 
   we will denote it $V$.

   Since $t\subseteq t^*\subseteq h^*$, 
   for all $(f^+,f^-) \in \mathcal{D}(t)$ 
   then necessarily,  as in the proof of \emph{\ref{thm:Hk.gamma=12}} of \Cref{thm:Hk.good},
   \begin{equation}\label{eq:nellaproof.lim}
     \lim_{r\to 0} 
     \begin{vmatrix}
       f^+(r) & \overline{f^+(r)} \\
       f^-(r) & \overline{f^-(r)} 
     \end{vmatrix}
     = 0.
   \end{equation}
   The equations \eqref{eq:nellaproof.lim} and \eqref{eq:det.gamma.<1/2} 
   imply that
   \begin{equation*}
     \begin{vmatrix}
       A^+ & \overline{A^+} \\
       A^- & \overline{A^-}
     \end{vmatrix}
     = 2i \Im(A^+ \overline{A^-})=0,
     \quad \text{ for all }(A^+,A^-) \in V.
   \end{equation*}
   Thanks to \Cref{lem:determinat.matrix.null}, 
   $V=V_\theta:=\seq*{(A^+,A^-)\in\C^2:\,A^+\sin\theta+A^-\cos\theta=0}$ for some
   $\theta \in [0,\pi)$
   or $V=\seq{0}$. This last case can not happen, since $t$ can not have
   proper symmetric extensions, being self-adjoint.
   In conclusion, all the self-adjoint extensions of $h$ are of the form $t(\theta)$ for $\theta\in [0,\pi)$, 
   and \eqref{eq:defn.Hk.<1/2} holds. 

   Conversely, we prove that for all $\theta \in [0,\pi)$ the operators $t(\theta)$ are self-adjoint.
   It is easy to check that they are symmetric and that they extend $h$.
   Let $(f^+,f^-) \in \mathcal{D}\left(t(\theta)^*\right)$: by the definition 
   there exists $(f^+_0,f^-_0) \in L^2(0,+\infty)^2$ such that 
   $\langle (f^+,f^-), t(\theta)(\widetilde{f}^+,\widetilde{f}^-) \rangle = \langle (f^+_0,f^-_0), (\widetilde{f}^+,\widetilde{f}^-)\rangle$ for all 
   $(\widetilde{f}^+,\widetilde{f}^-) \in \mathcal{D}\left(t(\theta)\right)$, and $(f^+_0,f^-_0) =t(\theta)^*(f^+,f^-)$. 
   Since $t(\theta) \subseteq t(\theta)^* \subseteq h^*$, 
   \begin{equation*}
   \begin{split}
     \langle h^*(f^+,f^-) , (\widetilde{f}^+,\widetilde{f}^-) \rangle_{L^2} &=     
     \langle t(\theta)^* (f^+,f^-),(\widetilde{f}^+,\widetilde{f}^-)\rangle_{L^2} =
     \langle (f^+_0,f^-_0) ,  (\widetilde{f}^+,\widetilde{f}^-) \rangle_{L^2} \\
     &=
     \langle (f^+,f^-) , t(\theta) (\widetilde{f}^+,\widetilde{f}^-) \rangle_{L^2} =
     \langle (f^+,f^-) , h^*(\widetilde{f}^+,\widetilde{f}^-) \rangle_{L^2},
     \end{split}
   \end{equation*}
   and this happens if and only if 
   \begin{equation}\label{eq:s.a.condition}
     \begin{vmatrix}
       A^+ & \overline{\widetilde A^+} \\
       A^- & \overline{\widetilde A^-}
     \end{vmatrix}
     =
     \lim_{r\to 0}    
     \begin{vmatrix}
       f^+(r) & \overline{\widetilde f^+(r)} \\
       f^-(r) & \overline{\widetilde f^-(r)}
     \end{vmatrix}
     =0,
   \end{equation}
where 
   \begin{equation*}
			\lim_{r\to 0} 
           \abs*{\begin{pmatrix} f^+(r) \\ f^-(r) \end{pmatrix} 
           - D 
           \begin{pmatrix}
             A^+ r^{\gamma} \\
             A^- r^{-\gamma}
           \end{pmatrix}}r^{-1/2}
           =0,
           \qquad
           \lim_{r\to 0} 
           \abs*{
           \begin{pmatrix} \widetilde f^+(r) \\ \widetilde f^-(r) \end{pmatrix} 
           - D 
           \begin{pmatrix}
             \widetilde A^+ r^{\gamma} \\
             \widetilde A^- r^{-\gamma}
           \end{pmatrix}}r^{-1/2}
           =0.
   \end{equation*}
           From \eqref{eq:s.a.condition},
           there exists $(a,b) \in \C^2$, $(a,b)\neq (0,0)$ such that
           $a(A^+,A^-) + b(\overline{\widetilde A^+},\overline{\widetilde A^-})=0$.  In particular,
           we choose $(\widetilde A^+,\widetilde A^-)\neq (0,0)$ in order to guarantee $a\neq 0$: we have that
           \[
           a(A^+ \sin \theta +A^- \cos \theta) + b(\overline{\widetilde A^+}\sin \theta+ \overline{\widetilde A^-}\cos \theta)=0
           \]
that implies $(A^+, A^-) \in V_\theta$, that is $(f^+,f^-) \in \mathcal{D}\left(t(\theta)\right)$.

\emph{\ref{thm:Hk.gamma=0}}
  The proof of this case is analogous to the one of \emph{\ref{thm:Hk.gamma<12}}, for this reason we will omit some details.
  Let $t$ be a self-adjoint extension of $h$. Then, thanks to\textit{ \ref{item:properties.domain.gamma=0}} of \Cref{thm:caract:d(h*)} 
   we have that for all $(f^+,f^-) \in \mathcal{D}(t)$
   there exists $(A^+,A^-) \in \C^2$
   such that 
   \begin{equation*}
\begin{split}
&\lim_{r\to 0}
\abs*{
    \begin{pmatrix}
    f^+(r) \\ f^-(r)
  \end{pmatrix}
  - (M\log r+\mathbb{I}_2)
  \begin{pmatrix}
    A^+  \\ A^-
  \end{pmatrix}
 } 
 r^{-1/2}
 = 0,
\end{split}
\end{equation*}
   where $M$ is the real matrix defined in \eqref{eq:defn.M}.
Let $V$ be the linear subspace of $\C^2$ defined as the image of the homomorphism $(f^+,f^-) \in \mathcal{D}(t)\mapsto (A^+,A^-)\in \C^2$. Since $t$ is symmetric, we get that for $(f^+,f^-) \in \mathcal{D}(t)$:
\begin{equation*}
 \lim_{r\to 0}    
     \begin{vmatrix}
       f^+(r) & \overline{f^+(r)} \\
       f^-(r) & \overline{f^-(r)}
     \end{vmatrix}
     =0,
\end{equation*}
and, thanks to \eqref{eq:det.gamma.=0}, it happens if and only if
\begin{equation*}
 \begin{vmatrix}
       A^+ & \overline{ A^+} \\
       A^- & \overline{ A^-}
     \end{vmatrix}
     =0.
\end{equation*}
Applying \Cref{lem:determinat.matrix.null} we deduce that $V=V_\theta=\seq*{(A^+,A^-)\in\C^2:\,A^+\sin\theta+A^-\cos\theta=0}$ for some $\theta\in [0,\pi)$, that is $t=t(\theta)$.

Conversely, let us prove that any $t(\theta)$ is self-adjoint. It is clearly symmetric and it extends $h$. Moreover, 
Let $(f^+,f^-) \in \mathcal{D}\left(t(\theta)^*\right)$: by the definition 
we get that for any $(\widetilde{f}^+,\widetilde{f}^-)\in \mathcal{D}\left(t(\theta)\right)$
\begin{equation}\label{eq:aggiunto.gamma=0}
\langle t(\theta)^*(f^+,f^-), (\widetilde{f}^+,\widetilde{f}^-) \rangle_{L^2} =
\langle (f^+,f^-),  t(\theta)(\widetilde{f}^+,\widetilde{f}^-) \rangle_{L^2}.
\end{equation}
Since $t(\theta)$ extends $h$, using the same notation of \emph{\ref{item:properties.domain.gamma=0}} of \Cref{thm:caract:d(h*)}, we can affirm that \eqref{eq:aggiunto.gamma=0} holds if and only if 
\begin{equation*}
 \begin{vmatrix}
       A^+ & \overline{\widetilde A^+} \\
       A^- & \overline{\widetilde A^-}
     \end{vmatrix}
     =0.
\end{equation*}
From this and thanks to the fact that $(\widetilde A^+, \widetilde A^-)\in V_\theta$ we deduce that $(A^+, A^-) \in V_\theta$, that is $(f^+,f^-)\in \mathcal{D}\left(t(\theta)\right)$.
\end{proof}  

For the proof of \Cref{thm:Hk.supercritical} we need the following Lemma.
 
\begin{lemma}\label{lem:determinat.matrix.null<0}
Let $V$ be a complex proper subspace of $\C^2$ and $\tau>0$. Then the following are equivalent:
\begin{enumerate}[label=(\roman*)]
\item\label{item:i.determinant<0}
$(A^+,A^-)\in V$ if and only if
$|A|=\tau |B|$;
\item\label{item:ii.determinant<0}
$V=\seq{(0,0)}$ or $V=V_\theta:=\langle(\tau e^{i\theta} ,e^{-i\theta})\rangle$
with $\theta\in[0,\pi)$.
\end{enumerate} 
\end{lemma}
\begin{proof}
We prove that \emph{\ref{item:i.determinant<0}} implies \textit{\ref{item:ii.determinant<0}}, since the other implication is obvious.
Let $V$ be as in\emph{ \ref{item:i.determinant<0}}: $V$ can not be the whole $\C^2$, 
so $V$ is a proper subspace of $\C^2$, i.e.~it has dimension zero or one. 
In the first case $V=\seq{(0,0)}$. Let us suppose now that $V$ has dimension one, 
that is $V=\langle(A^+_0,A^-_0)\rangle$ for some $(A^+_0,A^-_0)\neq (0,0)$ with $|A^+_0|=\tau |A^-_0|$. 
In radial coordinates we have $A^+_0 = c_1 e^{i a}, A^-_0= c_2 e^{i b}$ and
$ c_1 = \tau c_2 \neq 0$. 
Setting $(A^+,A^-):= c_2^{-1} e^{-i \frac{a+b}{2}} (A^+_0,A^-_0)=(\tau e^{i\theta},e^{-i \theta})$, with $\theta:=(a-b)/2$,
we have immediately the thesis, since $\langle(A^+,A^-)\rangle=\langle(A^+_0,A^-_0)\rangle$.
\end{proof}

\begin{proof}[Proof of \Cref{thm:Hk.supercritical}]
  The proof of this Theorem is 
analogous to the one of \emph{\ref{thm:Hk.gamma<12}} in \Cref{thm:Hk.subandcritical},
but we need to use \Cref{lem:determinat.matrix.null<0} in place of \Cref{lem:determinat.matrix.null}.
   \end{proof}

\section{Proof of \Cref{thm:distinguished.gamma<12} and \Cref{thm:distinguished.gamma=0}}
\label{sec:distinguished}
Let us give some useful instruments before starting the proof of \Cref{thm:distinguished.gamma<12} and \Cref{thm:distinguished.gamma=0}.
Let $a\in \R\setminus\seq{-1/2}$. For any $\varphi,\chi\in C^\infty_c(0,+\infty)$ we set
\begin{equation*}
\langle\varphi,\chi \rangle_{\mathcal{J}_a}:=\int_0^{+\infty} \partial_r(r^a\varphi(r))\overline{\partial_r(r^a\chi(r))} r^{-2a}\, dr.
\end{equation*}
Thanks to \eqref{eq:hardy.gamma<1/2} and \eqref{eq:hardy.gamma>1/2}, $\langle \cdot, \cdot \rangle_{\mathcal{J}_a}$ defines a scalar product on $C^\infty_c(0,+\infty)$. 
Therefore, if $||\cdot||_{\mathcal{J}_a}$ is the norm induced by $\langle \cdot, \cdot \rangle_{\mathcal{J}_a}$, we get that 
$\mathcal{J}_a:=\overline{C^\infty_c(0,+\infty)}^{||\cdot||_{\mathcal{J}_a}}$ is a Hilbert space.

Let $\varphi\in C^\infty_c(0,+\infty) $. Integrating by parts we get:
\begin{equation}\label{eq:quadrato.J}
||\varphi||^2_{\mathcal{J}_a}=\int_0^{+\infty}  \abs{\partial_r(r^a \varphi(r))}^2 r^{-2a} dr=\int_0^{+\infty} |\varphi'(r)|^2dr+a(a+1)\int_0^{+\infty}\frac{|\varphi(r)|^2}{r^2}\,dr.
\end{equation}
From \eqref{eq:quadrato.J} and thanks to \eqref{eq:hardy.gamma<1/2} and \eqref{eq:hardy.gamma>1/2} we deduce that
\begin{equation*}
\begin{array}{rclc}
(2a+1)^2 ||\varphi||^2_{\mathcal{J}_0}&\leq ||\varphi||^2_{\mathcal{J}_a}\leq&||\varphi||^2_{\mathcal{J}_0}&\text{if}\ a(a+1)\leq 0,\\
||\varphi||^2_{\mathcal{J}_0}&\leq ||\varphi||_{\mathcal{J}_a}\leq&(2a+1)^2 ||\varphi||^2_{\mathcal{J}_0}&\text{if}\ a(a+1)> 0;
\end{array}
\end{equation*}
that means that $\mathcal{J}_a=\mathcal{J}_0=:\mathcal{J}$.

\begin{lemma}\label{lem:caract.J}
Let $\mathcal{J}$ be defined as above. Then
\begin{equation*}
\mathcal{J}=\seq*{u\in AC[0,M]\ \text{for any}\ M>0:
\ \frac{u}{r},\ u'(r)\in L^2(0,+\infty)}.
\end{equation*}
\end{lemma}
\begin{proof}
We set $\tilde{\mathcal{J}}:=\seq*{u\in AC[0,M]\ \text{for any}\ M>0: \ \frac{u}{r},\ u'(r)\in L^2(0,+\infty)}$. 

Let us prove that $\mathcal{J}\subset\tilde{\mathcal{J}}$: let $\left(u_n\right)_n\subset C^\infty_c(0,+\infty)$ be a Cauchy-sequence in $||\cdot||_{\mathcal{J}}$. Thanks to \eqref{eq:hardy.gamma<1/2} we get that for any $n,m\in \N$
\begin{equation*}
||u_m-u_n||_{\mathcal{J}}^2=
\int_0^{+\infty} 
\abs*{u'_m(r)-u'_n(r) }^2\, dr
\geq 
\frac{1}{4}
\int_0^{+\infty}
\frac{\abs{u_m(r)-u_n(r)}^2}{r^2}\, dr,
\end{equation*}
that means that $\left(\frac{u_n}{r}\right)_n\subset C^\infty_c(0,\infty)$ is a Cauchy-sequence in $L^2$. Let 
$u$ and $\tilde{u}$ be such that 
$
\frac{u_n}{r}\to \frac{u}{r}$ in $L^2$
and
$
u'_n \to \tilde{u}$ in $ L^2$.
Moreover, $u'_n\to u'$ in the sense of distribution. By the uniqueness of the limit we deduce that $u'=\tilde{u}$ and so $u\in \tilde{\mathcal{J}}$.

To prove that $\tilde{\mathcal{J}}\subset\mathcal{J}$ we will follow the approach of \cite[Section 4]{esteban2017domains}.
Let $u\in \tilde{\mathcal{J}}$ and firstly assume that its support is a compact subset of $(0,+\infty)$.
Let $\left(\varphi_n\right)_n$ be a sequence of mollifier function and set $u_n:=\varphi_n* u$. By construction $\left(u_n\right)_n\subset C^\infty_c(0,+\infty)$ and $u_n\to u$ in $\mathcal{J}$, that gives $u\in \mathcal{J}$. 
Let us finally assume that the support of $u$ is not compact. We set 
\begin{equation*}
\eta(r):=
\begin{cases}
0&\text{if} \ 0\leq r\leq 1,\\
r-1&\text{if} \ 1\leq r\leq 2,\\
1&\text{if} \ 2\leq r,
\end{cases}
\quad
\text{and}
\quad
\zeta(r):=
\begin{cases}
1&\text{if} \ 0\leq r\leq 2,\\
-r+3&\text{if} \ 2\leq r\leq 3,\\
0&\text{if} \ 3\leq r,
\end{cases}.
\end{equation*}
Finally, for any $n\in\N$, we set $\eta_n(r):=\eta (nr)$, $\zeta_n(r):=\zeta\left(\frac{r}{n}\right)$ and $u_n:=\left(\eta_n+\zeta_n\right) u$.
For any $n\in \N$, $u_n\in \mathcal{J}$ because its support is compact by construction and
$u_n\in \tilde{\mathcal{J}}$. Indeed:
$u_n\in AC[0,M]$ for any $M>0$ and $\frac{u_n}{r}\in L^2$ because the support of $u_n$ is compact.
Moreover
$u'_n=(\eta_n+\zeta_n)u'+(\eta_n+\zeta_n)'u\in L^2$
because, on the right-hand side, both are $L^2$ function on compact subsets of $(0,+\infty)$.

Finally \begin{align*}
||u_n-u||_{\mathcal{J}}^2 &
\leq 2 \int_0^{+\infty}|(\eta_n(r)+\zeta_n(r))u'(r)-u'(r)|^2\, dr+2 \int_0^{+\infty}|(\eta_n(r)+\zeta_n(r))'u(r)|^2\, dr\\
&=:I_1(n)+I_2(n).
\end{align*}  
Regarding the first term we see that
\begin{equation*}
I_1(n)\leq 2 \int_{0}^{2/n}|u'(r)|^2\, dr+2 \int_{2n}^{+\infty}|u'(r)|^2\,dr\to 0,
\end{equation*}
if $n\to+\infty$, by the dominated convergence theorem.
About the second term we notice
\begin{equation*}
I_2(n)=2n^2 \int_{2/n}^{3/n}|u(r)|^2\,dr+\frac{2}{n^2}\int_{2n}^{3n} |u(r)|^2\, dr\leq 8\int_0^{3/n}\frac{|u(r)|^2}{r^2}\, dr+18\int_{2n}^{+\infty}\frac{|u(r)|^2}{r^2}\,dr\to 0,
\end{equation*}
if $n\to+\infty$, by the dominated convergence theorem.
Then $u_n\to u$ in $\mathcal{J}$ that gives $u\in \mathcal{J}$.
\end{proof}

We can now start the proofs of \Cref{thm:distinguished.gamma<12} and \Cref{thm:distinguished.gamma=0}. From now on we will use the same notation in \eqref{eq:no.m_jk_j}.
\begin{proof}[Proof of \Cref{thm:distinguished.gamma<12}]
We start proving that
\emph{\ref{lab:distinguished.gamma<12ii}}$\Rightarrow$\emph{\ref{lab:distinguished.gamma<12iii}}. 
Let $\theta = 0$. Then, for any $(f^+,f^-)\in \mathcal{D}\left(t(0)\right)$ there exists $A^+\in \C$ such that
\begin{equation*}
  \int_0^{+\infty} 
\frac{1}{r^2}\abs*{
  \begin{pmatrix}
    f^+(r) \\ f^-(r)
  \end{pmatrix}
  - D 
  \begin{pmatrix}
    A^+ r^\gamma \\ 0
  \end{pmatrix}
}^2
\,dr.
\end{equation*}
that tells us that 
\begin{equation*}
\int_0^{+\infty}
\frac{\abs{f^+(r)-B^+ r^\gamma}^2}{r^2}\, dr
+
\int_0^{+\infty}
\frac{\abs{f^-(r)-B^-r^\gamma}^2}{r^2}\, dr
<+ \infty
\end{equation*}
with 
$\begin{pmatrix}
B^+\\
B^-
\end{pmatrix}=
D\begin{pmatrix}
A^+\\
0
\end{pmatrix}
$.
Since $0<\gamma<1/2$ we deduce that for $a\in \left[\frac{1}{2},\frac{1}{2}+\gamma\right)$,
\begin{equation*}
\int_0^{+\infty}
\frac{|f^\pm(r)|^2}{r^{2a}}\,dr\leq
2\int_0^1
\frac{\abs{f^\pm(r)-B^\pm r^\gamma}^2}{r^2}\,dr
+2|B^\pm|
\int_0^1
r^{2\gamma-2a}\,dr+
\int_1^{+\infty}
|f^\pm(r)|^2\,dr<+\infty.
\end{equation*}

It is trivial that \emph{\ref{lab:distinguished.gamma<12iii}} implies \emph{\ref{lab:distinguished.gamma<12iv}}.

Let us now show that \emph{\ref{lab:distinguished.gamma<12iv}} implies \emph{\ref{lab:distinguished.gamma<12ii}}.
Let $\theta \in [0,\pi)$ and
$(f^+,f^-) \in \mathcal{D}\left(t(\theta)\right)$,
such that \eqref{eq:defn.Hk.<1/2} holds for $A^\pm\in\C$ and assume that $f^\pm\in \mathcal{D}(r^{-1/2})$.

Let $\varphi^-$ be defined as in \eqref{eq:phi=Mf}. Therefore $\varphi^-\in \mathcal{D}(r^{-1/2})$ and \eqref{eq:cond.C} holds. 
Then
\begin{equation*}
\begin{split}
\int_0^1 \frac{\abs{A^- r^{-\gamma}}^2}{r}\,dr   \leq & 
  2\int_0^1 \frac{\abs{\varphi^--A^- r^{-\gamma}}^2}{r^2}\,dr
  +
  2\int_0^1 \frac{\abs{\varphi^-}^2}{r}\,dr 
  < +\infty.
\end{split}
\end{equation*}
Since $0<\gamma<1/2$, we conclude that $A^-=0$. From the arbitrariness of $(f^+,f^-) \in \mathcal{D}\left(t(\theta)\right)$, 
we have $\theta=0$.

To conclude the proof it remains to show that \emph{\ref{lab:distinguished.gamma<120}} and \emph{\ref{lab:distinguished.gamma<12ii}} are equivalent.
Let $(f^+,f^-)\in \mathcal{D}\left(t(\theta)\right)$ and $(A^+,A^-)\in\C^2$ such that \eqref{eq:defn.Hk.<1/2} holds.

We notice that $\varphi^-_{m_j,k_j}$ defined in \eqref{eq:def.phi^-.gamma<12} and $\varphi^-$ defined in \eqref{eq:phi=Mf} coincide. Then, from \eqref{eq:cond.C}, we deduce that
 $\varphi^-\in \mathcal{J}$ if and only if $A^-=0$ that is equivalent to say that $\theta=0$ due to the arbitrariness of $(f^+,f^-)\in\D(t(\theta))$.
\end{proof}

\begin{proof}[Proof of \Cref{thm:distinguished.gamma=0}]
Let $(f^+,f^-)\in \mathcal{D}\left(t(\theta)\right)$ and $(A^+,A^-)\in\C^2$ such that \eqref{eq:defn.Hk.=0} holds. 
In the case that $\nu+\mu\neq0$ we notice that $\varphi^-_{m_j,k_j}$ 
defined in \eqref{eq:def.phi^-.gamma=0} and $\varphi^-$ defined in \eqref{eq:phi=Mf} coincide. 
From \eqref{eq:cond.C.gamma=0}, we deduce that
 $\varphi^-\in \mathcal{J}$ if and only if 
$
(\nu+\mu)A^++(k+\lambda)A^-=0.
$
Due to the arbitrariness of $(f^+,f^-)$ it is equivalent to say that $\theta$ is as in \emph{\ref{lab:distinguished.gamma=0iii}}. 

Let us assume $\nu+\mu=0$.
Then $\varphi^-_{m_j,k_j}=-2\nu f^-\in \mathcal J$ if and only if $A^-=0$ that is equivalent to say $\theta=0$ due to the arbitrariness of $(f^+,f^-)\in \D(t(\theta))$.
\end{proof}

\end{document}